\documentclass[a4paper,twoside,10pt,leqno,
]{amsart}
\usepackage{amsmath,amsthm,amssymb,enumerate,url}
\usepackage[mathscr]{eucal}

\swapnumbers
\theoremstyle{plain}
\newtheorem{theorem}{Theorem}[section]
\newtheorem{lemma}[theorem]{Lemma}
\newtheorem{GSS}[theorem]{The  Stallings-Swan Theorem}

\newtheorem{stallings}[theorem]{Stallings' Ends Theorem}
\newtheorem{as}[theorem]{The Almost Stability Theorem}

\newtheorem{corollary}[theorem]{Corollary}
\newtheorem{proposition}[theorem]{Proposition}

\theoremstyle{definition}

\newtheorem{definitions}[theorem]{Definitions}
\newtheorem{example}[theorem]{Example}

\newtheorem{notation}[theorem]{Notation}

\newtheorem{remark}[theorem]{Remark}

\newtheorem{remarks}[theorem]{Remarks}

\numberwithin{equation}{theorem}

\newcommand{\abs}[1]{\left\lvert#1\right\rvert} 
\newcommand{\gen}[1]{\langle\mkern3mu#1\mkern3mu\rangle}

\def \bbl1 {[\mkern-3mu[}
\def \bbr1 {]\mkern-3mu]}

\DeclareMathOperator{\size}{size}
\DeclareMathOperator{\complete}{Complete}
\DeclareMathOperator{\substabs}{-substabs}
\DeclareMathOperator{\irr}{irr}
\DeclareMathOperator{\length}{length}
\DeclareMathOperator{\Almosts}{Almosts}
\DeclareMathOperator{\V}{V}
\DeclareMathOperator{\E}{E}
\DeclareMathOperator{\T}{T}

\DeclareMathOperator{\U}{U}
\DeclareMathOperator{\X}{X}

\DeclareMathOperator{\Map}{Maps}
\DeclareMathOperator{\PP}{\mathscr{P}}
\DeclareMathOperator{\BB}{\mathscr{B}}
\DeclareMathOperator{\SP}{\Sigma P\hskip-.5pt}
\DeclareMathOperator{\rank}{rank}

\def \G {\mathfrak{ast}}
\def \comp {\textbf{c}}
\def \Ord {\mathcal{O}\mkern-1mu\mathit{rd}}
\def \leftmod {{\setminus}}

\def \reals {\mathbb{R}}

\def \integers {\mathbb{Z}}
\def \EE {\mathscr{E}}

\def \AA {\mathscr{A}}
\def \FF {\mathscr{F}}

\def\d1{\discretionary{-}{}{-}}
\def\coloneq{\mathrel{\mathop\mathchar"303A}\mkern-1.2mu=}

\def\ttt{\mathbf{t}}

\renewcommand{\phi}{\varphi}
\renewcommand{\epsilon}{\varepsilon}
\renewcommand{\le}{\leqslant}
\renewcommand{\ge}{\geqslant}
\renewcommand{\P}{\operatorname{P}}

\begin{document}

\title[An improved proof of the Almost Stability Theorem]
{An improved proof of\\the Almost Stability Theorem}

\author[Warren Dicks]{Warren Dicks\hskip1pt*}
\thanks{*\,Research supported by MINECO (Spain) through project numbers MTM2014-53644-P and  MTM2017-83487-P}
\date{}

\begin{abstract}
 In 1989, Dicks and Dunwoody proved the Almost Stability Theorem, which
has among its corollaries the Stallings-Swan theorem that groups of
cohomological dimension one are free. In this article,
we  use  a nestedness result of Bergman, Bowditch, and Dunwoody to simplify somewhat
the  proof of the finitely generable   case of the Almost Stability Theorem.
We also simplify the proof of the non finitely generable case.

The  proof we give here of the Almost Stability Theorem is essentially self contained, except that  in the non
finitely generable case  we refer the reader to the original argument for
the proofs of two technical lemmas about groups acting on trees.

\medskip

{\footnotesize

}

\end{abstract}

\maketitle

\begin{flushright}
\textit{To Martin Dunwoody%
, on the occasion of his 80th birthday.
}\end{flushright}

\bigskip

\section{Introduction}

Throughout, $G$ will denote a discrete, multiplicative group.

Unexplained terminology and notation used in the first two sections
will be defined in Section~\ref{sec:terminology}.

\begin{definitions} \label{defs:6.1}
For any sets $E$ and $Z$, we write $\Map(E,Z)$ to
denote  the  set of all  maps of sets from $E$ to $Z$,
with each $v \in \Map(E,Z)$ written as  $v\colon E \to Z$, $e \mapsto \langle v,e \rangle.$
For any  \mbox{$v, w \in \Map(E,Z)$}, we write\vspace{-1mm}
$$v \triangledown w:= \{e \in E \mid  \langle v,e \rangle \ne  \langle w,e \rangle\};\vspace{-1mm}$$
if this set  is  finite, then we say that
$v$ and $w$ are \textit{almost equal},  and write \mbox{$v =_{\text a} w$}.
Almost equality is an equivalence relation on $\Map(E,Z)$; its
  equivalence classes are called \textit{almost equality classes}.

If $E$ and $Z$ are  (left) $G$-sets, then $\Map(E,Z)$ is   a  $G$-set, with the
conjugation $G$-action, that is, if
\mbox{$v \in \Map(E,Z)$}, \mbox{$g \in G$}, and  \mbox{$e \in E$},  then
\mbox{$\langle  gv, e \rangle \coloneq g\langle   v, g^{-1}e \rangle$}, and,  hence,
 \mbox{$ g\langle   v, e \rangle = \langle  g v, g e \rangle$.}

A  $G$-set is said to be \textit{$G$-free}  if
each element's $G$-stabilizer is trivial, and is said to be \textit{$G$-quasifree} if
each element's $G$-stabilizer is finite.
 \hfill\qed
\end{definitions}

The following  is one form of~\cite[III.8.5]{DD}; see Remarks~\ref{rems:1}(ii) below.

\begin{as}\label{thm:ast} If  $E$ and $Z$ are any $G$-sets such that
$E$ is $G$-quasifree and  each element's $G$-sta\-bi\-li\-zer  stabilizes  some element of~$Z$,
then each $G$-sta\-ble almost equality class in   $\Map(E,Z)$
is the vertex $G$-set of some $G$-tree.  Any such $G$-tree  automatically has $G$-quasifree edge $G$-set.
\end{as}

The purpose of this article is to give a revised proof of Theorem~\ref{thm:ast} that
incorporates various simplifications which have become available since the original proof was published.

Let $\BB(G)$ denote the Boolean algebra of al\-most right $G$-stable subsets of $G$.
For $G$ finitely generated,
Bergman~\cite{GMB1} defined a well\d1ordered  measure on $\BB(G)$,
and
Bowditch and Dunwoody~\cite[8.1]{BHB} used the well-orderedness of Bergman's measure to show
that  each Boolean   $G$-subalgebra  of $\BB(G)$
 is generated by some nested  $G$-subset.
We shall recall their proofs, and then use their results to simplify the proof of
the case of Theorem~\ref{thm:ast} where $G$ is finitely generable.
Although it closely follows part of the proof in~\cite{DD}, this argument had not
been recorded before now;
Bowditch and Dunwoody~\cite[14.2]{BHB} had noted
the weaker conclusion that
 each $G$-sta\-ble almost equality class \emph{embeds} in the vertex $G$-set
of some $G$-tree, in the case where $G$ is finitely generable and $Z= \integers/2\integers$.

In the complementary case, where $G$ is not finitely generable,
we shall describe some further simplifications in the proof of Theorem~\ref{thm:ast}.

The  proof we give  of the Almost Stability Theorem~\ref{thm:ast} is essentially self contained, except that  in the non
finitely generable case  we refer the reader to the original argument in~\cite{DD} for
the proofs of two technical lemmas about groups acting on trees.

The article has the following structure.

In Section~\ref{sec:2}, to provide motivation,
we digress to show that Theorem~\ref{thm:ast} yields one   form of the result
of Stallings~\cite[6.8]{Stallings1} and Swan~\cite{Swan}  that groups of cohomological dimension one are free.

In Section~\ref{sec:3}, we record rather a large number of definitions, which will provide
much of the basic terminology that we shall be using.

In Section~\ref{sec:4}, we recall  from~\cite{MJD} and~\cite{DD}
Dunwoody's construction of trees from nested sets, here with a simplification by Roller~\cite{Roller}.

In Section~\ref{sec:5}, we recall  from~\cite{GMB1}  Bergman's well-ordered
measure, and we recall from~\cite{BHB}  the Bowditch-Dunwoody construction of
nested generating sets.

 In Section~\ref{sec:BB}, we use the results of Section~\ref{sec:4}
and the nested sets of Section~\ref{sec:5}
 to construct a  tree,
and we deduce  a result from~\cite{DD} which strengthened  a result of Dunwoody~\cite{MJD}.

In Section~\ref{sec:ends}, to provide motivation, we digress to deduce one
  form of Stallings' Ends Theorem~\cite[4.1]{Stallings2}.

In Section~\ref{sec:6}, we recall from~\cite{DD} the deduction of the finitely generable case of
Theorem~\ref{thm:ast}  from the results of Section~\ref{sec:BB}.

We then consider  the non finitely generable case, closely following ~\cite{DD}
 but with an improved
transfinite induction procedure.

In Section~\ref{sec:8}, we record, without proofs,  two lemmas about trees proved in~\cite{DD}.

In Section~\ref{sec:9}, we fix  notation that applies for the remainder of the proof.

In Section~\ref{sec:10}, we give results and proofs about finitely generable extensions.

In Section~\ref{sec:11}, we give  results and proofs about countably generable extensions.

In Section~\ref{sec:12}, we give the proof of the general case.

In Section~\ref{sec:13}, we give the proof of  the   analogue for extensions.

In this article,  we shall work with trees, and not discuss Bass-Serre theory.
We shall mention in each of the two digressions that certain information about trees
may be translated by Bass-Serre theory into information about groups.

\begin{remarks}\label{rems:1}  Let $E$ and $Z$ be any $G$-sets, and
 $ V$ be any  $G$-stable almost equality class in the $G$-set  $\Map(E,Z)$.

\medskip

 (i). We denote by  $\complete(V)$ the $G$-graph
with vertex $G$-set $V$ and edge $G$-set $\{(v,w) \in V \times V \mid  v \ne w\}$,
where each edge $(v,w)$ has initial vertex $v$ and terminal vertex $w$;
here, the $G$-stabilizer of   $(v,w)$ is a
subgroup of the $G$-stabilizer of $v \triangledown w$,
 and $v \triangledown w$ is a finite, nonempty subset of~$E$.
Thus, if $E$ is  $G$-quasifree, then the edge $G$-set of $\complete(V)$ is also   $G$-quasifree,
and, in particular, any $G$-tree with vertex $G$-set~$V$ has $G$-quasifree edge $G$-set.

\medskip

(ii).  Consider the following conditions.
\begin{enumerate}[\hskip .9cm(a)]
\item $G$ stabilizes each element of $Z$.
\item  $E$ is such that each element's  $G$-sta\-bi\-li\-zer   stabilizes  some element of~$Z$;
equivalently, there exists some
$G$-map from $E$ to $Z$;
equivalently, $G$~stabilizes some element of $\Map(E,Z)$.
\item Each finite subgroup of $G$ stabilizes some element of $V$.
\setcounter{enumi}{25}
\item  $E$ is  $G$-quasifree.
\end{enumerate}

Notice that  (b) and (z) are the hypotheses in Theorem~\ref{thm:ast}.
Since equivalence classes are nonempty by definition,  $V \ne \emptyset$, and  $\Map(E,Z)\ne \emptyset$;
 hence, if $E \ne \emptyset$, then $Z \ne \emptyset$.
It is easy to see that \mbox{(a) $\Rightarrow$ (b)}
and that  \mbox{(c)+(z) $\Rightarrow$ (b)}.
It is not difficult to use properties of almost equality to prove that    (b) implies~(c).
Thus,  if (z) holds, then \mbox{(b) $\Leftrightarrow$ (c) }.

In 1989, Dicks and Dunwoody~\cite[III.8.5]{DD}   proved the  case of Theorem~\ref{thm:ast}
 where (a)  holds.
 In this article, we
  shall see that    (b), as opposed to (a),  is the condition that was
used in that proof.

 Since (c) is a \textit{necessary} condition for the $G$-set $V$ to be the vertex $G$-set of a $G$-tree,
we now see that Theorem~\ref{thm:ast} says that
if (z) holds, then
the $G$-set $V$~is the vertex $G$-set of some $G$-tree if and only if (b) holds.

\medskip

 (iii).  In Theorem~\ref{thm:ast}, each hypothesis on $E$ determines a corresponding condition on
$\complete(V)$, and we have the following formulation:  If $E$ is  $G$-quasifree,
then, first, the edge $G$-set of $\complete(V)$ is   $G$-quasifree,
and, secondly, the $G$-set $\Map(E,Z)$ has some $G$-stable element if and only if
the $G$-set consisting of the maximal subtrees of $\complete(V)$
has some  $G$-stable  element.

 In the simplest case, where $V$ is  the almost equality class of a $G$-stable element~$v$ of $\Map(E,Z)$,
 there exists a  $G$-stable maximal subtree  of $\complete(V)$ with
 edge set $\{v\} \times (V{-}\{v\})$. \qed
 \end{remarks}

\section{Digression 1: The   Stallings-Swan Theorem}\label{sec:2}

In this section, to motivate  interest in the Almost Stability Theorem,
we recall how it implies the Stallings-Swan result that groups of
cohomological dimension one are free.
This and many other applications may be found in~\cite[Chapter~IV]{DD}.

  Let $\integers G$ denote  the integral group ring, and $\omega \integers G$ denote its augmentation ideal.

In 1953, Fox~\cite[(2.3)]{Fox} proved, but did not state, that if the group $G$ is   free,  then
the left $\integers G$-module $\omega \integers G$ is free. In 1956,
this implication was made explicit
by Cartan and Eilenberg~\cite[X.5]{CE}, who further observed that if $G$ is a nontrivial free group,
then the projective dimension of the left $\integers G$-module  $\integers$   is equal to  1.
In 1957, Eilenberg and Ganea~\cite{EG} defined  `the dimension  of a group~$G$',
now called the \textit{cohomological dimension of $G$}, to be
the projective dimension of the left $\integers G$-mod\-ule~$\integers$.  Thus, by definition, $G$ has
cohomological dimension at most one if and only if the
left $\integers G$-module $\omega \integers G$ is projective.  Hence, by Fox's result, all
 free groups have cohomological dimension at most one.
Eilenberg and Ganea remarked that they did not know whether or not all groups of
cohomological dimension   one are free.
In 1968, Stallings~\cite[6.8]{Stallings1} proved that all finitely generable groups of
cohomological dimension  one are free; in 1969,  Swan~\cite{Swan}
proved that \textit{all} groups of cohomological dimension one are free.
In the academic year within this same  period, 1968--9, Serre
gave a course on what is now called \textit{Bass-Serre theory},
and one of the many new results presented was the fact that
the group $G$ is free   if and only if $G$ acts freely on some
tree~\cite[I.3.2.15 and I.3.3.4]{Serre}.
To my knowledge, neither direction had  previously been stated in the literature in exactly this form.
It seems plausible that Dehn knew the `only if' direction in 1910,
in the context of  his work on Cayley graphs in~\cite{Dehn}.
Reidemeister came close to knowing the `if' direction  in 1932
in the context of his tree-based proof in~\cite[4.20]{Reidemeister} of the Nielsen-Schreier theorem  that
all subgroups of free groups are free.

In summary then, the left $\integers G$-module  $\omega \integers G$ is
projective if and only if
the group $G$ is  free   if and only if $G$ acts freely on some  tree.

The currently known proofs of the Almost Stability Theorem~\ref{thm:ast} use  many of the arguments of Stallings
and Swan.   We  shall now recall that one form of their theorem
 is in turn a consequence of Theorem~\ref{thm:ast}.

\begin{GSS}\label{thm:proj} If the left $\integers G$-module $\omega \integers G$ is projective, then
$G$ acts freely on some  tree.
\end{GSS}

\begin{proof}
 By hypothesis,  there
 exists some left $\integers G$-module $Q$ such that the left
$\integers G$-module $\omega \integers G \oplus Q$ is  free.
There then exists some free left $\integers$-module $A$ such that
the  (free) left $\integers G$-modules
$\omega \integers G \oplus Q$ and $AG:= \integers G \otimes_\integers A$ are isomorphic, and may be identified.
In a natural way, $\Map(G,A)$ is a left $\integers G$-module, and we may
 identify $AG$ with the ($G$-stable) almost equality class of   $0$  in  $\Map(G,A)$;
here, it is to be understood that  $G$ stabilizes each element of $A$.
Each element $r$ of $AG$ has a unique expression as $p+q$ with $p \in \omega \integers G$ and $q \in Q$,
and here  we shall write $r = p\oplus q$.

Let $g$ and $x$ represent variable   elements of $G$ ranging over all of $G$.

We set\vspace{-3mm}
$$\widehat{g{-}x}:= (g{-}x)\oplus 0 \in \omega \integers G \oplus Q = AG \subseteq \Map(G,A).$$
Notice that \mbox{$\widehat{g{-}x}=_{\text a}0$} in $\Map(G,A)$  and   \mbox{$\langle \widehat{g{-}x}, x \rangle \in A$}.
Essentially following Specker~\cite{Specker}, we consider the element $v$ of \mbox{$ \Map(G,A)$}
defined by  $ \langle v, x \rangle \coloneq \langle  \widehat{1{-}x}, x \rangle$,
and show that $gv   =   \widehat{g{-}1} + v =_{\text a}0+v$ in $\Map(G,A)$ as follows:\vspace{-1mm}
 \begin{align*}  &\langle gv , x \rangle
  =    g\langle v, g^{-1}x \rangle
\overset{\text{def  }v}{=}  g\langle  \widehat{1{-}g^{-1}x}, g^{-1}x \rangle
 \\&  = \langle  \widehat{g{-}x}  ,  x \rangle
=
\langle  \widehat{g{-}1} +  \widehat{1{-}x},x \rangle  \overset{\text{def  }v}{=}
  \langle  \widehat{g{-}1}+v  ,x \rangle.
\end{align*}
 Hence, the almost equality class $v + AG$ in $\Map(G,A)$  is $G$-stable.  By  the
Almost Stability Theorem~\ref{thm:ast},
the $G$-set $v + AG$ is the vertex $G$-set of some  $G$-tree.

It remains to show that the vertex $G$-set \mbox{$v + AG$} is $G$-free.
Suppose then that we have some \mbox{$g \in G$} and some \mbox{$r  \in AG$} such that
\mbox{$g(v+r) = v+r$}  in $\Map(G,A)$,
  that is, \mbox{$ (1{-}g)r = (g{-}1)v = \widehat{g{-}1}$}. Write \mbox{$r = p\oplus q$} with
\mbox{$p \in \omega \integers G$} and \mbox{$q \in Q$}.
Then   \mbox{$(1-g)p\oplus(1-g)q=(g-1)\oplus 0 $}  in \mbox{$\omega \integers G \oplus Q$}.   Thus,
\mbox{$g(p+1) = p+1$} in \mbox{$\integers G \subseteq \Map(G,\integers)$}.
Hence, \mbox{$p+1$} is constant on each   \mbox{$\langle g \rangle$}-orbit  \mbox{$\langle g \rangle x$} in $G$.
Since    \mbox{$p \in \omega \integers G$}, \mbox{$0 \ne p+1 =_{\text a} 0$} in $\Map(G,\integers)$.
Hence, $  \langle g \rangle  $ is finite.
  Now the augmentation map carries \mbox{$p+1$} to a $\integers$-multiple of
\mbox{$ \vert \langle g \rangle \vert$} and also to $1$.  Thus,
\mbox{$ \vert \langle g \rangle \vert = 1$}.  Hence,  $g = 1$, as desired.
\end{proof}

\begin{remark}  The foregoing argument applies to give Dunwoody's
characterization
of the groups $G$ such that the  left $RG$-module $\omega RG$ is projective, where $R$ is any
nonzero associative ring with $1$
and  $\omega RG$ denotes the augmentation ideal of the group ring $RG$; see~\cite[1.1]{MJD},~\cite[IV.3.13]{DD}. \qed
\end{remark}

\section{Terminology}\label{sec:terminology}\label{sec:3}

In this section, we collect together definitions of many of the concepts that we shall be using.

\begin{notation}  We  write $f \vert_D$ to indicate the map obtained from a map $f$  by restricting
the domain of $f$ to a subdomain $D$.

We shall find it useful to have notation  for intervals in $\integers$
that is different from the notation for intervals in $\reals$.
Let $i$, $j \in \integers$.
We define the sequence $$\bbl1 i{\uparrow}j\bbr1 \coloneq
 \begin{cases}
(i,i+1,\ldots, j-1, j) \in \integers^{j-i+1}   &\text{if $i \le j$,}\\
() \in \integers^0 &\text{if $i > j$.}
\end{cases}
$$
The subset of $\integers$ underlying $\bbl1 i{\uparrow}j\bbr1 $ is denoted $[i{\uparrow}j]\coloneq \{i,i+1,\ldots, j-1, j \}$.

We set $\bbl1 i{\uparrow}\infty\bbl1 \,\,\, \coloneq (i,i+1,i+2,\ldots)$ and
$[i{\uparrow}\infty[ \,\,\, \coloneq \{i,i+1,i+2,\ldots\}$ .

Suppose we have a set $V$ and a map $[i{\uparrow}j] \to V$, $\ell \mapsto v_\ell$.
We define the corresponding sequence in $V$ by $$v_{\bbl1 i{\uparrow}j \bbr1 }\coloneq  \begin{cases}
(v_i,v_{i+1},   \cdots, v_{j-1},   v_j) \in V^{i-j+1}  &\text{if $i \le j$,}\\
() \in V^0 &\text{if $i > j$.}
\end{cases}$$
By abuse of notation, we shall also express this sequence   as $(v_\ell \mid \ell \in \bbl1 i{\uparrow}j\bbr1 )$,
although \mbox{``$\ell \in \bbl1 i{\uparrow}j\bbr1 $"} on its own will not be assigned a meaning.
The set of terms of~$v_{\bbl1 i{\uparrow}j \bbr1 }$ is denoted~$v_{[i{\uparrow}j]}$.

We set $v_{\bbl1 i{\uparrow}\infty\bbl1 \,} \coloneq (v_i,v_{i+1},v_{i+2},\ldots)$
and $v_{[i{\uparrow}\infty[ \,} \coloneq \{v_i,v_{i+1},v_{i+2},\ldots\}$.
\hfill\qed
\end{notation}

\begin{definitions} By a \textit{well-ordered set},
we mean a set $S$ together with a total order~$\sqsubset$
such that, for each nonempty subset $T$ of $S$, there exists some
 $x \in T$ such that,  for each $t \in T$,
  $x \sqsubseteq t$.  It is then usual to treat the total order as  ``less than", and to
use phrases such  as ``all strictly descending  sequences are finite".

An \textit{ordinal} is a set $\beta$ such that, first, each element of $\beta$ is equal to some subset of
$\beta$, and,  secondly, $\beta$ is well-ordered by $\in$\,; see~\cite[2.10]{Jech}.

The three lower-case Greek letters $\alpha$, $\beta$, and $\gamma$  will be used to denote ordinals.

We let $\Ord$ denote the class of all ordinals, and,
for \mbox{$\alpha$, $\beta \in \Ord$,} we define  \mbox{$\alpha < \beta $} to mean \mbox{$\alpha \in \beta$}.
Thus, for each \mbox{$\beta \in \Ord$},  \mbox{$\beta = \{ \alpha \in \Ord  \mid \alpha < \beta\}$.}

Let $S$ be any set.  By the axiom of choice, $S$ can be well-ordered, and, hence,
there exists some $\alpha \in \Ord$ such that there exists some bijective map of sets from $\alpha$ to $S$.
The minimum of the set consisting of  such $\alpha$ is   denoted   $\vert S \vert$.
We write $\omega_{0}\coloneq \bigl\vert\,\, [0{\uparrow}\infty[\,\,\bigr\vert$; thus, $\omega_0$ is
the smallest infinite ordinal, the set of finite ordinals.
By abuse of notation, we view the elements of $[0{\uparrow}\infty[$\,\,  as  finite ordinals.  \hfill\qed
\end{definitions}

\begin{definitions}\label{defs:0}  Let $V$ be any set.

We denote by $\PP(V)$  the set of all subsets of~$V$,
and view $\PP(V)$ as a Boolean algebra in the usual way.

Let $A$ and $B$ be any elements of  $\PP(V)$.

We write \mbox{$A^{\comp} \coloneq \{v \in V \mid v \not \in A\}$}.

We  say that $A$ and $B$ are \textit{nested}
   if \mbox{$\emptyset \in \{ A \cap B, A \cap B^\comp ,  A^\comp \cap B, A^\comp \cap B^\comp\}$}.

We write $A{-}B \coloneq A \cap B^\comp$.

We write $A\vee B$ to denote   $A \cup B$ in the situation where $A \cap B = \emptyset$.

We write $A\triangledown B \coloneq (A{-}B)\vee (B{-}A)$.
If $ A\triangledown B$ is a finite set, we say that $A$ and $B$ are \textit{almost equal}, and write $A =_{\text a} B$.

For any subset $\EE$ of $\PP(V)$,
we denote by $\gen{\EE}_{{\BB}}$  the Boolean subalgebra of $\PP(V)$
generated by  $\EE$.

For each $v \in V$, we write $v^{\ast\ast} \coloneq \{A \in \PP(V) \mid v \in A\} \in \PP(\PP(V))$.
\hfill\qed
 \end{definitions}

\begin{definitions}\label{defs3.4}  We define the  \textit{rank of the group $G$} by
$$\rank(G)  \coloneq \min \{\,\vert S \vert :  S \text{ is a subset of $G$  which generates  } G \}.$$
For any subgroup $H$ of $G$, we define the  \textit{rank of $G$ relative to $H$}  by
$$\rank (G \text{\normalfont { rel }} H) \coloneq \min \{\,\vert S \vert :  S \text{ is a subset of $G$ such that $S\cup H$   generates  }
G \}.$$

By a \textit{$G$-set,} we mean a set $V$ together with a map
\mbox{$G \times V \to V, \,\, (g,v) \mapsto gv,$}
such that, for each  $v \in V$,
 $1v = v$ and, for each $(g_1, g_2) \in G \times G$, \mbox{$g_1(g_2v) = (g_1g_2)v$}.

By a \textit{right $G$-set}, we mean a set $V$ together with a map
\mbox{$V \times G \to V$},\, $(v,g) \mapsto vg$,
such that, for each  $v \in V$,
 $v1 = v$ and, for each $(g_1, g_2) \in G \times G$, \mbox{$(vg_1)g_2 = v(g_1g_2)$}.
(Here, $V$ is a $G$-set with $gv:= vg^{-1}$.)
All concepts defined for $G$-sets are understood to have analogues for right $G$-sets.

Consider  any $G$-set  $V$. For each subset $W$ of $V$, we write
$$GW \coloneq \{gw \mid g \in G, w \in W\}.$$
If   $GW = W$, we say that $W$ is a \textit{$G$-stable subset of $V$} and  a \textit{$G$-subset of $V$}.
For each $v \in V$, we define the \textit{$G$-stabilizer of $v$}
 to be \mbox{$G_v \coloneq \{g \in G \mid gv = v\} \le G$},
and the \textit{$G$-orbit of $v$} to be $Gv \coloneq \{gv \mid g \in G\}$, a $G$-subset of $V$.
We let \mbox{$G \leftmod V \coloneq \{Gv \mid v \in V\}$}, a  partition of $V$.
We say that $V$ is
\textit{$G$-finite} if  $G \leftmod V  $ is finite.
A \textit{$G$-transversal  in $V$} is a subset of $V$
which contains exactly one element of each $G$-orbit of $V$.
A subgroup~$H$ of~$G$ is said to  \textit{stabilize} an element $v$ of $V$ if
 $Hv = \{v\}$ or, equivalently,
$H_v = H$ or, equivalently, $H \le G_v$; if $H$ stabilizes some element of~$V$,
 we say that $H$ is a
\textit{$G$-substabilizer for~$V$}.
We let $G\substabs(V)$  denote the set of $G$-substabilizers for $V$.

Consider  any $G$-sets $V$ and $W$.  By a \textit{$G$-map} $\phi\colon V \to W$, $v \mapsto \phi(v)$, we mean a
map of sets such that, for each $(g,v) \in G \times V$, $\phi(gv) = g\phi(v)$.  There exists some
 $G$-map from $ V$ to $W$ if and only if  $G\substabs(V) \subseteq G\substabs(W)$.

A $G$-set $V$ is said to be \textit{$G$-incompressible} if each self $G$-map of $V$ is
bijective, or, equivalently, for each $(v,w) \in V \times V$, if $G_v \le G_w$, then
$G_v = G_w$ and $Gv = Gw$.  If $V$ is not $G$-incompressible, we say that $V$ is
\textit{$G$-compressible}.

We now consider the  right $G$-set $G$. Any  $A\in\PP(G)$ is  said to be
 \textit{almost right $G$-stable}  if,
for each $g \in G$, $Ag =_{\text a} A$.
We write  $\BB(G)$ to denote the Boolean subalgebra of $\PP(G)$
consisting of all the almost right $G$-stable elements.
For each subgroup $H$ of $G$,  any   $A \in \PP(G)$ is
 said to be  \textit{almost a right $H$-set}
if $A$ is almost equal to some right $H$-subset of $G$.
For each subset $\EE$ of $\PP(G)$, we let
$\Almosts(\EE)$  denote the set consisting of all those subgroups $H$ of~$G$
which have the property that each element of
$\EE$ is almost a right $H$-set.
\hfill\qed
\end{definitions}

\begin{definitions}\label{defs3.5} By a \textit{graph} $X$, we mean a quadruple
 $(\V(X),\E(X), \iota_{X},\tau_{{X}})$
  where
$\V(X)$ and $\E(X)$ are two disjoint  sets and
$\iota_{X}$ and $\tau_{{X}}$ are  maps from $\E(X)$ to~$\V(X)$.
Where $X$ is clear from the context, we write $\iota$ for $\iota_{X}$ and
$\tau$ for $\tau_{{X}}$.
We define $\vert X \vert \coloneq \vert \V(X) \vee \E(X) \vert$.
We say that   $\V(X)$ is the \textit{vertex set} of~$X$
and that $\E(X)$ is the \textit{edge set} of $X$, and  that
$\iota $ and $\tau $ are the \textit{incidence maps} of $X$.
We say that the elements of $\V(X)$ are the \textit{vertices} of $X$,
and the elements of $\E(X)$ are the \textit{edges} of~$X$.
For each edge $e$ of $X$, we say that $e$ is \textit{incident to} $\iota e $ and $\tau e $,
and that $\iota e $ is the
\textit{initial vertex} of $e$ and that $\tau e $ is the \textit{terminal vertex} of $e$.

A \textit{$G$-graph} $X$ is
  a graph for which  $\V(X)$, $\E(X)$ are  $G$-sets, and
$\iota $, $\tau $ are $G$-maps.
Passing to $G$-orbits gives a \textit{quotient graph $G\leftmod X$}.
Here $X$ is \textit{$G$-finite} if
$ G\leftmod X $ is finite, that is, $ \vert G\leftmod X \vert$ is finite.

For any subset $S$ of $G$,  the \textit{Cayley graph} $\X(G,S)$ is defined as the $G$-graph with
vertex $G$-set $G$, edge $G$-set $G \times S$ with $G$-action \mbox{$g_1(g_2,s):= (g_1g_2,s)$},
and incidence maps assigning to each edge \mbox{$(g,s) \in G \times S$} the initial vertex  $g$ and
the terminal vertex  $gs$.

Let $X$ be any graph.

A \textit{subgraph of $X$} is a graph whose vertex set and edge set are subsets of the
vertex set and edge set of $X$, respectively, and whose incidence maps agree  with those of $X$.

For each vertex $v$ of $X$, the  \textit{valence}
of $v$ in $X$ is $$\vert \{e \in \E(X) : \iota e = v \} \vert\,\, +\,\, \vert \{e \in \E(X) : \tau e = v \}\vert.$$
We say that $X$ is \textit{locally finite} if each vertex's valence is finite.

We create a set $\E^{-1}\mkern-2mu(X)$ together with a bijective map
$\E(X) \to \E^{-1}\mkern-2mu(X)$, \mbox{$e \mapsto e^{-1}$,} called \textit{inversion}.
We define $\E^{\pm 1}\mkern-2mu(X) \coloneq  \E(X) \vee \E^{-1}\mkern-2mu(X)$. We extend
$\iota $ to a map  $\iota  \colon\E^{\pm 1}\mkern-2mu(X) \to \V(X)$
by setting $\iota (e^{-1}) \coloneq \tau e $  for each $e \in \E(X)$.
Similarly, we extend
$\tau $ to a map   \mbox{$\tau  \colon \E^{\pm 1}\mkern-2mu(X)\to \V(X)$}
by setting $\tau (e^{-1}) \coloneq \iota e $ for each \mbox{$e \in \E(X)$.}
We extend inversion to a map $\E^{\pm 1}\mkern-2mu(X)\to\E^{\pm 1}\mkern-2mu(X)$, $e \mapsto e^{-1}$,
by defining \mbox{$(e^{-1})^{-1} \coloneq e$} for each $e \in \E(X)$.

By an \textit{$X$-path}, we shall mean any sequence   \mbox{$p=(v_0, e_1, v_1, e_2, \ldots, e_n,v_n)$}
such that \mbox{$n \in [0{\uparrow}\infty[\,$},
 $v_{\bbl1 0{\uparrow}n\bbr1 }$ is a sequence in $\V(X)$, $e_{\bbl1 1{\uparrow}n \bbr1 }$
is a sequence in $\E^{\pm 1}\mkern-2mu(X)$, and, for each $i \in [1{\uparrow}n]$,
$\iota (e_i) = v_{i-1}$ and $\tau (e_i) = v_i$.
We define the \textit{inverse} of $p$ to be
\mbox{$p^{-1} \coloneq (v_n, e_n^{-1}, \ldots, e_2^{-1}, v_1, e_1^{-1}, v_0)$.}
We say that $p$ is \textit{reduced} if, for each $i \in [2{\uparrow}n]$, $e_{i} \ne e_{i-1}^{-1}$.
We say that $p$ \textit{joins} $v_0$ to $v_n$, and that the pair $(v_0, v_n)$ is \textit{$X$-joined}.
  We define $\length(p) \coloneq n$.
If there exists no  $X$-path joining  $v_0$ to $v_n$  of smaller
length,  then we say that the \textit{$X$-distance} between $v_0$ and $v_n$ is~$n$.
For each finite subset $S$ of $\E(X)$,
we define the \textit{number of times $p$   crosses~$S$}
to be $\vert\{i \in [1{\uparrow}n] : e_i \in  S^{\pm 1}\}\vert$;
if this number is positive, we say that \textit{$p$   crosses~$S$}.
Where $S$  consists of a single edge, we shall usually speak of paths crossing that edge rather than crossing~$S$.

We say that $X$ is \textit{connected} if each pair of vertices  of $X$ is $X$-joined.
The maximal nonempty connected subgraphs of $X$ are called the \textit{components of $X$}.

For any subset $E'$ of $\E(X)$, the  \textit{graph obtained from $X$ by collapsing $E'$},
denoted $X/E'$, is the graph with edge set $E'^{\comp} \coloneq \E(X)-E'$,
vertex set the set of components of $X - E'^{ \comp}$, and
the induced incidence maps.  For example, $X/\E(X)$ maps bijectively to the set of components of $X$,
and here every edge of $X$ gets collapsed.

For each   $A \in \PP(\V(X))$,  we define the \textit{coboundary of $A$} (in $X$) as
\begin{align*}
\delta_X (A)  \coloneq \,\,&
\{e\in \E(X) \mid
(\iota e , \tau e ) \in \bigl((A  \times A^\comp) \vee (A^\comp \times A)\bigr)\}
\\ = \,\,& \{e\in \E(X) \mid
A \,\,\in \,\, (\iota e )^{\ast\ast}\, \,\triangledown\, (\tau e )^{\ast\ast}\};
\end{align*}
 where $X$ is clear from the
context we write  \mbox{$\delta A$} in place of \mbox{$\delta_X (A)$}.

The   \textit{Boolean algebra of $X$}, denoted $\BB(X)$, is defined as
the Boolean subalgebra of $\PP(\V(X))$
consisting of all the elements with finite coboundary in $X$.

We say that $X$ is a \textit{tree} if, for each $(v,w) \in  \V(X) \times \V(X)$, there exists
a unique reduced $X$-path  that joins $v$ to $w$. A \textit{$G$-tree}  is a $G$-graph which is a tree.
We say that $X$ is a \textit{forest} if, for each $(v,w) \in  \V(X) \times \V(X)$, there exists
at most one reduced $X$-path  that joins $v$ to $w$.
A \textit{$G$-forest}  is a $G$-graph which is a forest.

Let $T$ be any $G$-tree.
 We say that $T$ is   \textit{$G$-incompressible} if the $G$-set
$\V(T)$ is  $G$-incom\-pressible.
An edge $e$ of $T$ is said to be \textit{$G$-compressible} if
 there exists some $(v,w) \in \{(\iota e, \tau e), (\tau e, \iota e)\}$
 such that $Gv \ne Gw$ and $G_v \le G_w$; here, $G_w = G_e$.
\hfill\qed
\end{definitions}

In the following, the important conclusion $\BB(X) = \BB(G)$  is due to Specker~\cite{Specker}.

\begin{lemma}\label{lem:equalB}
Let $S$ be any  generating set of $G$, and set  \mbox{$X \coloneq \X(G,S)$.}
Then  $X$ is a  nonempty, connected,    $G$-free $G$-graph, and $\V(X)= G$.
Moreover, if $S$ is finite, then  $X$ is    locally finite, $X$ is $G$-finite,   and $\BB(X) = \BB(G)$.
\end{lemma}

\begin{proof} Clearly, $X$ is a  nonempty, $G$-free $G$-graph, and $\V(X)= G$.
Also,  $X/\E(X)$ is a $G$-set with one $G$-orbit,
and the image of $1$ in  $X/\E(X)$ is stabilized by $S$.   Since $S$ generates $G$,  we see that
the image of $1$ in  $X/\E(X)$  is stabilized by $G$.  Hence, $X$ has exactly one component.

Now suppose that $S$ is finite.  Then $X$ is   locally finite and $G$-finite.
It remains to verify  that $\BB(X) = \BB(G)$.

Consider any $A  \in \PP(G)$.  For each $ s   \in   S$,
\begin{align*}
A \triangledown (As^{-1}) &= \{g \in G \mid g \in A - As^{-1}  \text { or } g \in As^{-1} -A \}
\\&= \{g \in G \mid g \in A, gs \in A^\comp \text { or } g \in A^\comp, gs \in A\},
\end{align*}   and, hence,
\begin{equation}\label{eq:cay}
A \triangledown (As^{-1})  =  \{g \in G \mid (g,s) \in \delta(A) \}.
\end{equation}

Suppose that $A \in \BB(G)$.  For each element $s$ of the finite set $S$,
 $ A \triangledown (As^{-1})$ is a finite set.  It follows from~\eqref{eq:cay} that $\delta(A)$ is a finite set,
that is, $A \in \BB(X)$.

Suppose that $A \in \BB(X)$.  For each $s \in S$, by~\eqref{eq:cay}, $ A \triangledown As^{-1} $ is a finite set,
that is, \mbox{$A =_{\text a} As^{-1}$.}  Since $S$ generates $G$, it follows that $A \in \BB(G)$.
\end{proof}

\section{Building trees from nested sets}\label{sec:4}

This section reviews results of Dunwoody~\cite{MJD} with modifications by Dicks and Dunwoody~\cite{DD}  and Roller~\cite{Roller}.

\begin{notation}\label{not:1}   Let $V$ be  any set,   and  $\EE$  be any subset
   of the Boolean algebra $\PP(V)$.

We say that $\EE$ is \textit{$\comp$-stable} if, for each $A \in  \EE$, we have $A^\comp \in \EE$.

We say that $\EE$ is \textit{finitely separating} if, for all $v$,  $w \in V$,
$v^{\ast\ast} \cap \EE =_{\text a} w^{\ast\ast} \cap \EE$;
note that
 $  v^{\ast\ast} \cap \EE  = \{d \in \EE \mid v \in d\}$.

We say that $\EE$ is \textit{nested} if, for each $(e,f) \in \EE \times \EE$,
$e$ and $f$ are nested in~$V$, that is,
\mbox{$\emptyset \in \{e \cap f, e \cap f^{\,\comp}, e^\comp \cap f, e^\comp \cap f^{\,\comp}\}$}.

 For each $e \in \EE$, we define
$$
\iota  e  \coloneq \{ d  \in \EE \mid d  \supseteq  e \text { or }
d  \supset  e^\comp\},
 \tau e \coloneq \{e,e^\comp\},  \text{ and }
\tau' e \coloneq  \{ d  \in \EE \mid d  \supset   e \text { or }
d  \supseteq  e^\comp\}.$$

If $\EE \ne \emptyset$, we define  $\T(\EE)$ to be the graph for which the edge set is $\EE$,
the vertex set is \mbox{$\{\iota e, \tau  e \mid e \in \EE\} \subseteq \PP(\PP(V))$},
and each $e \in \EE$ has initial
vertex  $\iota e$ and  terminal vertex  $\tau e$.

If $\EE \ne \emptyset$, we define
$\U(\EE)$ to be the graph for which the edge set is $\EE$, the vertex set is
 \mbox{$\{\iota e, \tau'  e \mid e \in \EE\} \subseteq \PP(\PP(V))$},
and  each $e \in \EE$ has initial
vertex  $\iota e$ and  terminal vertex  $\tau' e$.

If $\EE=\emptyset$,  we define both $\T(\EE)$ and $\U(\EE)$ to be the  graph for which the
  edge set is the empty set $\EE$ and the
vertex set is \mbox{$\{ \EE \} \subseteq \PP(\PP(V))$}.
\hfill\qed
\end{notation}

\begin{example} Let $T$ be a tree.  It is sometimes natural to think of the vertices of~$T$
as certain sets of edges of~$T$,
and it is sometime natural to think of the edges of~$T$ as certain sets of vertices of~$T$;
to achieve this formally, we create `double duals' of the edges of~$T$.
For each $e \in \E(T)$, we set
$$e^{\ast\ast} \coloneq \{v \in \V(T) \mid \text{the reduced $T$-path   from  $v$ to $\tau e$ crosses $e$}\};$$
then $e^{\ast\ast}$ is the vertex set of that component of $T{-}\{e\}$ which contains $\iota e$; hence,
  $\delta_T(e^{\ast\ast}) = \{e\}$; hence, $e^{\ast\ast} \in \BB(T)$.
Set \mbox{$\EE(T) \coloneq \{e^{\ast\ast} \mid e \in \E(T)\}$}.  For each \mbox{$e \in \E(T)$},
\begin{align*}
 \iota_{\,\U(\EE(T))}(e^{\ast\ast})
&=\{ d^{\,\ast\ast}  \in \EE(T) \mid d^{\,\ast\ast}  \supseteq  e^{\ast\ast} \text { or }
d^{\,\ast\ast}  \supset  (e^{\ast\ast})^\comp\}
\\&= \{d^{\,\ast\ast} \in \EE(T) \mid \iota_T e \in d^{\,\ast\ast}\} = (\iota_T e)^{\ast\ast} \cap \EE(T).
\end{align*}
There is a natural identification $T = \U(\EE(T))$.\hfill\qed
\end{example}

The following is due to Dunwoody~\cite[2.1]{MJD} with modifications from~\cite[II.1.5]{DD}.
The proof given here incorporates the approach of Roller~\cite{Roller}.

\begin{theorem}\label{thm:Dunwoody} With {\normalfont Notation~\ref{not:1}},
if $V$ is  any set and $\EE$ is any  $\comp$-stable,  finitely separating, nested subset of~$ \PP(V) $
such that $\emptyset \not\in\EE$, then  the following hold.
\begin{enumerate}[\normalfont (i)]
\item \label{it:2} $\T(\EE)$ is a  tree with edge  set $\EE$, and,
 for any \mbox{$e, f \in \EE$},   the
$\T(\EE)$\d1distance between $\iota e$ and $\iota f$    equals
$\abs{(\iota e)\triangledown(\iota f)}$.
\item \label{it:3} There exists a natural map $V \to \V(\T(\EE))$, $v \mapsto v^{\ast\ast}\cap \EE$.
In detail,  if $\EE \ne \emptyset$, then, for each $v \in V$\!,
there exists some  $\subseteq$-minimal element $e$ of $v^{\ast\ast}\cap \EE$, and then
$v^{\ast\ast}\cap \EE  = \iota e$.
\end{enumerate}
\end{theorem}

\begin{proof} The case where $\EE = \emptyset$ is straightforward, and we shall assume that $\EE \ne \emptyset$.

 \eqref{it:2}.  Here, $\U(\EE)$ is the graph for which the edge set is $\EE$,
the vertex set is \mbox{$\{\iota e \mid e \in \EE\}$}, and each $e \in \EE$  has initial vertex $\iota e$ and
terminal vertex $\tau' e = \iota \hskip.5pt(e^\comp)$.
For each \mbox{$e \in \E(\U(\EE)) = \EE$,} \mbox{$e^\comp \ne e$},
\mbox{$e^{\comp\comp} = e$},  \mbox{$\iota ( e^\comp) = \tau (e)$}, and \mbox{$\tau (e^\comp) = \iota (e)$}.
By a    \textit{restricted $\U(\EE)$-path,}
we shall mean any sequence   \mbox{$p=e_{ \bbl1 1{\uparrow}n\bbr1 }$} in $\EE$
such that $n \in [1{\uparrow}\infty[\,$  and,
for each $i \in [2{\uparrow}n]$,
\mbox{$\iota (e_i) = \tau (e_{i-1})$}  and $e_{i} \ne  e_{i-1}^\comp$.

For each $e \in \EE$,
$(\iota e)\triangledown(\iota (e^\comp)) = \{e,e^\comp\}$,
and  $\iota e$  is a $\comp$-transversal in $\EE$, or `orientation', in the sense of a subset
 $\mathcal{O}$ of $\EE$ such that
 $\EE = \mathcal{O}\, \vee  \{f^\comp \mid f\in \mathcal{O}\}$.

Consider any $e$, $f \in \EE$.
Set $[e,f[\,\,\, \coloneq \{d \in \EE  \mid e \subseteq d \subset f\}$,
and   define $[e,f]$,\,\, $]e,f]$ and $]e,f[$\, analogously.
These sets are finite by the finitely separating condition,
since we may choose $v\in e$ and $w \in f^\comp$, and find that $[e,f] \subseteq v^{\ast\ast} {-} w^{\ast\ast}$.
We write $e \prec f$ to mean  $[e,f[\,\,  = \{e\}$ or, equivalently, $]e,f]  = \{f\}$.
Now
\begin{align*}
\iota e &=\{ d  \in \EE \mid e  \subseteq  d \text { or } e^\comp  \subset  d \}, \qquad
\EE - \iota  f  = \{ d \in \EE \mid  d  \subset  f \text { or } d  \subseteq  f^\comp\},
\end{align*}
and also $e=f$ or $e \subset f$ or $e \subseteq f^\comp$ or
 $e^\comp \subset f$ or $e^\comp \subset  f^\comp$.
We then see that $$\iota e - \iota f = \iota e \cap (\EE - \iota  f )
=  [e,f[ \,\,\,\cup\,\,\, [e, f^\comp]\,\,\,\cup\,\,\,\,\,\,]e^\comp,f[ \,\,\,\cup\,\,\, ]e^\comp,f^\comp],$$
and that
the latter union is empty if and only if  $e=f$ or $e^\comp \prec f$.
Thus, $ \iota e \subseteq \iota f$ if and only if $e=f$ or $e^\comp \prec f$.
Since $([e^\comp, f[\,)^\comp = [f^{\,\comp}, e[\,$, the condition $e^\comp \prec f$
is invariant under interchanging $e$ and $f$, and we see that
$\iota e=\iota f$ if and only if    $e=f$ or $e^\comp \prec f$.
By interchanging $e$ and~$e^\comp$, we see that $\tau' e=\iota f$ if and only if
$f = e^\comp$ or   $e \prec f$.  (This paragraph
is based on the elegant presentation of Roller \cite[$\S2.6$-$\S2.7$]{Roller}
and is simpler than the discussion in  \mbox{\cite[II.1.5]{DD}}.)

A  restricted $\U(\EE)$-path \mbox{$p= e_{\bbl1 1{\uparrow}n\bbr1 }$,
$n \in [1{\uparrow}\infty[\,$,}
may then be viewed as an unrefinable increasing
 sequence
\mbox{$e_1 \prec e_2 \prec \cdots \prec e_n$} in $\EE$.
Since $e_1 \subseteq e_n$, neither
$e_1 = e_n^\comp$ nor $e_n \prec e_1$ are possible; thus,  $\tau' e_n \ne \iota e_1$.
Hence,  in $\U(\EE)$,  no vertex is joined to itself by a restricted $\U(\EE)$-path.

We shall now see that,  in $\U(\EE)$, any vertex is joined to any other vertex by a  restricted $\U(\EE)$-path.
By the nestedness of $\EE$, for any   $e, f \in \EE$,
there exist \mbox{$e' \in \{e,e^\comp\}$} and \mbox{$f' \in \{f, f^\comp\}$} such that \mbox{$e' \subseteq f'$}.
Since the set $[e',f']$ is finite, there exists some
 unrefinable increasing
 sequence
$$e'=e_1 \prec e_2 \prec \cdots \prec e_n=f' \text{ in } \EE,   n \in [1{\uparrow}\infty[\,;$$ this gives
a  restricted $\U(\EE)$-path which meets the vertices of $e$ and $f$, as desired.

We may pass from the  graph  $\U(\EE)$ to the  graph $\T(\EE)$ by
detaching each edge from its terminal vertex and giving the elements of
each unordered pair of edges $\{e,e^\comp\}$, $e \in \EE$,
a new common terminal vertex.  Hence, $\T(\EE)$ is a tree.

The $\T(\EE)$-distance  formula follows since, for each \mbox{$e \in \EE$},
 \mbox{$(\iota e)\triangledown(\iota (e^\comp)) = \{e,e^\comp\}$} and
 $(\iota e, e, \tau e, (e^\comp)^{-1}, \iota (e^\comp))$
is a reduced $\T(\EE)$-path.

\eqref{it:3}.
We show first that $v^{\ast\ast} \cap \EE$ has $\subseteq$-minimal elements.
Since $\EE \ne \emptyset$, there exists some $f \in \EE$.
We may assume
that~$v \in f$, for otherwise we may replace $f$ with \mbox{$f^{\,\comp}$}.
 Since $\emptyset \not\in \EE$,  there exists some   \mbox{$w \in f^{\,\comp}$}, and we have
$$f \in \{ e \in \EE \mid v \in e  \subseteq f \}
\subseteq \{ e \in \EE \mid v \in e, w \in e^\comp\} = (v^{\ast\ast}\cap \EE)-(w^{\ast\ast} \cap \EE).$$
The latter set is finite, since $\EE$ is finitely separating.
Thus \mbox{$ \{ e \in \EE \mid v \in e \subseteq f \}$} is finite and nonempty,
and hence has a $\subseteq$-minimal element, which  is then a $\subseteq$-minimal element of
  \mbox{$\{ e \in \EE \mid v \in e \}$}, as desired.

 Let $e$ be a $\subseteq$-minimal element of $v^{\ast\ast} \cap \EE$.  We shall show   that $v^{\ast\ast} \cap \EE = \iota e$.

Let $d \in \iota e$.  Then either $d \supseteq e$ or $e \supset d^{\,\comp}$.
If $d \supseteq e$ then $d \supseteq e \supseteq \{v\}$ and, hence,
$d \in v^{\ast\ast} \cap \EE$.   If $e \supset d^{\,\comp}$, then, by the $\subseteq$-minimality of $e$,
$v \in (d^{\,\comp})^\comp = d$, and, hence  $d \in v^{\ast\ast} \cap \EE$.
Thus, $\iota e \subseteq  v^{\ast\ast} \cap \EE$.

Conversely, suppose that $d \in \EE - \iota e$.  Then $d^{\,\comp} \in \iota e \subseteq v^{\ast\ast} \cap \EE$.
Hence,    $d  \in \EE - v^{\ast\ast}$. Thus, $\EE - \iota e \subseteq  \EE - v^{\ast\ast}$.

Now $v^{\ast\ast} \cap \EE = \iota e$, as desired.
\end{proof}

\begin{corollary}\label{cor:Dunwoody} Let $V$ be  any set,   and
 $\EE $ be any   finitely separating, nested subset of~$ \PP(V) $
such that $\emptyset \not\in\EE $, $V \not\in\EE $, and, for each $e \in \EE $, $e^\comp \not\in \EE $.
With {\normalfont Notation~\ref{not:1}},  the following hold.
\begin{enumerate}[\normalfont (i).]
\item \label{it:2a} $\U(\EE)$ is a  tree with edge  set $\EE$, and,
 for any $v$, $w \in \EE$,   the
$\U(\EE)$-distance between $v$ and $w$    equals
$\abs{v\triangledown w}$.
\item \label{it:3a} There exists a natural map $V \to \V(\U(\EE))$, $v \mapsto v^{\ast\ast}\cap \EE$.
\end{enumerate}
\end{corollary}

\begin{proof} Set \mbox{$\EE^\comp \coloneq \{e^\comp \mid e \in \EE\}$.}
By Theorem~\ref{thm:Dunwoody}, $\T(\EE\vee \EE^\comp)$ is a tree with edge set
$\EE\vee \EE^\comp$. Now \mbox{$\T(\EE\vee \EE^\comp)/\EE^\comp = \U(\EE)$,} and the result follows.
(Alternatively, \mbox{$\U(\EE) = \U(\EE \vee \EE^\comp) - \EE^\comp $}, the tree  obtained from
  $\U(\EE \vee \EE^\comp)$ by choosing the  orientation  $\EE$.)
\end{proof}

\section{Building nested sets from graphs}\label{sec:5}

We now review   theory developed by Bergman in~\cite{GMB1}.

\begin{definitions}\label{defs:10}
We introduce a new symbol $\ttt$, and view the power-series ring~$\integers[[\ttt]]$
 as an ordered abelian group
 with the total order~$\sqsubset$  such that
$$\textstyle\sum\limits_{\ell\in[0{\uparrow}\infty[} c_\ell \ttt^\ell
\, \,\,\sqsubset\,  \sum\limits_{\ell\in[0{\uparrow}\infty[} d_\ell \ttt^\ell$$
if and only if there exists some \mbox{$\ell_0 \in [0{\uparrow}\infty[\,$}
such that \mbox{$c_{\ell_0} < d_{\ell_0}$} and, for each
\mbox{$\ell \in [0{\uparrow} \ell_0 [\,$},
\mbox{$c_\ell = d_\ell$}. We view the polynomial ring $\integers[\ttt]$ as a subset of~$\integers[[\ttt]]$.

 Let $X$  be any connected, locally finite graph.

For any set   $P$ of $X$-paths with the property that, for each $\ell \in [0{\uparrow}\infty[\,$,
 the set $P_\ell := \{p \in P : \length(p) = \ell\}$  is finite,
 we  write
$$\Sigma(P) \coloneq  \textstyle\sum\limits_{p \in P} \ttt^{\length(p)}
= \sum\limits_{\ell \in [0{\uparrow}\infty[} \abs{P_\ell} \ttt^\ell
\in \integers[[\ttt]].$$

For any element   $A$ of $\BB(X)$,
we let $\P(A)$ denote the set of all $X$-paths which begin in $A$ and end in $A^\comp$, necessarily crossing
 $\delta A$.
Since $X$ is locally finite and $\delta A$ is finite,  we see that
$\P(A)$ has only finitely many elements of any given length.
We write
\mbox{$\SP(A) \coloneq \Sigma(\P(A)).$}  Inversion of paths carries
$\P(A)$ bijectively to $\P(A^\comp)$; hence,
 \mbox{$\SP(A)   = \SP(A^\comp).$}
We write $$\SP(\BB(X)) \coloneq \{\SP(A) \mid A \in \BB(X)\} \subseteq \integers[[\ttt]].$$

For any   Boolean  subalgebra  $\AA$ of $\BB(X)$, any element $C$ of  $\AA$
  is said to be \textit{$\AA$-reducible}
if \vspace{-2mm} $$C \in \gen{ \{D \in \AA : \SP(D)\, \sqsubset \,\SP(C) \}}_{{\BB}};$$
otherwise, $C$ is said to be \textit{$\AA$-irreducible}.
We let $\irr(\AA)$ denote the set of all $\AA$-irreducible elements of~$\AA$.
Notice that $\emptyset$ and $\V(X)$ are   $\AA$-reducible.
\hfill\qed
 \end{definitions}

The following is the $G$-finite case of a result of Bergman~\cite[Lemma~2]{GMB1}.

\begin{theorem}\label{th:berg3} If $X$ is any connected, locally finite, $G$-finite $G$-graph,
 then $\SP(\BB(X))$ is a well-ordered subset of \,$\integers[[\ttt]]$.
\end{theorem}

\begin{proof} We shall show that a larger subset of $\integers[[\ttt]]$ is well-ordered.

Let $\mathbb{S}$   denote
the set of all finite subsets of $\E(X)$.

Consider any \mbox{$S \in \mathbb{S}$}.  We denote by $\P(S)$   the set of all
those $X$-paths that  cross $S$  an \textit{odd} number of times.
For each $\ell \in [0{\uparrow}\infty[$\,, we denote by $\P_{\ell}(S)$   the set of all elements of $\P(S)$
whose length equals~$\ell$.
Since $S$ is finite and $X$ is locally finite,  $\P_\ell(S)$ is finite; $\vert\P_\ell(S)\vert$ is an even number since
$\P_\ell(S)$  is stable under path inversion.  Clearly,
$\vert \P_0(S) \vert = 0$ and $\vert \P_1(S) \vert = 2 \vert S \vert$.
   We  write
 $$\SP(S) \coloneq \Sigma(\P(S))=
 \textstyle\sum\limits_{\ell \in [0{\uparrow}\infty[} \abs{\P_{\ell}(S)}\ttt^\ell
\,\,\text{and} \,\,
 \SP(\mathbb{S}) \coloneq \{\SP( S ) : S\in \mathbb{S}\} \subseteq 2\ttt\integers[[\ttt]].\vspace{-1mm}$$

For any $A \in \BB(X)$,  we have $\delta A \in \mathbb{S}$,    and  $\P(\delta A)$,
the set  of $X$-paths that cross $\delta A$ an odd number of times,
equals \mbox{$\P(A ) \vee \P(A^\comp )$}; hence,
\mbox{$\SP(\delta A) = 2\SP(A)$.}  Thus, \mbox{$2 \SP(\BB(X)) \subseteq \SP(\mathbb{S})$}, and it suffices to
show that \mbox{$\SP(\mathbb{S})$}
 is  well\d1ordered.

Consider any map \mbox{$S_{-}:[0{\uparrow}\infty[ \,\,\to \,\mathbb{S}$}, \mbox{$n \mapsto S_n$},
such that the composite map
\mbox{$\SP(S_{-}):[0{\uparrow}\infty[ \,\,\to \,\integers[[\ttt]]$}, \mbox{$n \mapsto \SP(S_n)$},
is decreasing.    It suffices to show
 that there exists some infinite subset $\mathbf{N}$ of $[0{\uparrow}\infty[$\,
such that \mbox{$\{\SP(S_n) \mid n \in \mathbf{N}\}$} has exactly one element.
Without loss of generality, we may assume that, for each $n \in [0{\uparrow}\infty[\,$,   $S_n \ne \emptyset$.

Let $\mathbf{K}$ denote the set consisting of those $k \in [1{\uparrow}\infty[$\,
for which there exist \vspace{.5mm}
\begin{align}
&\text{an infinite subset $\mathbf{N}$ of  $[0{\uparrow}\infty[\,$,
a map  $[1{\uparrow}k] \to \mathbb{S} {-}\{\emptyset\}$, $i \mapsto R_i$, and a map}\label{eq:k}
\\[-2mm]&\text{\mbox{$\mathbf{N} \times [1{\uparrow}k] \to G$}, $(n,i)  \mapsto g_{n,i}$, such that,
for  each $n \in \mathbf{N}$,
  $S_{n} = \textstyle\bigcup\limits_{i=1}^{k} g_{n,i}R_i$.}\nonumber
\end{align}
In the case where $k=1$,
$$\{\SP(S_n) \mid n \in \mathbf{N}\} = \{\SP(g_{n,1}R_1) \mid n \in \mathbf{N}\} = \{\SP(R_1) \},$$
which gives the desired result.
We shall show that $\mathbf{K}\ne\emptyset$
and  that,  for each
$k\in\mathbf{K}$,  either $k=1$ or $k{-}1\in\mathbf{K}$. This implies that \mbox{$1 \in \mathbf{K}$},
which completes the proof.

We now show that \mbox{$\vert S_0 \vert \in \mathbf{K}$}, and, hence, \mbox{$\mathbf{K}\ne\emptyset$}.
Let us choose a finite $G$-transversal $R$ in the $G$-finite $G$-set~$\E(X)$.
Consider any \mbox{$n \in [0{\uparrow}\infty[$\,}.  Since \mbox{$n \ge 0$} and $\SP(S_{-})$  is decreasing,
\mbox{$\SP(S_n) \,\sqsubseteq\, \SP(S_0)$}.  Hence,
$$  2\vert S_n\vert  \ttt \,\sqsubseteq\, \SP(S_n) \,\sqsubseteq\, \SP(S_0) \,\sqsubset\,
2 (\vert S_0\vert +1) \ttt.$$
Thus, \mbox{$ \vert S_n \vert <   \vert S_0 \vert + 1$}.
Set \mbox{$k:= \vert S_0 \vert$}.  Then \mbox{$1 \le \vert S_n \vert \le k$},
and we may choose a surjective map
\mbox{$[1{\uparrow}k] \to S_n$},
\mbox{$i \mapsto s_{n,i}$}. For  each \mbox{$i \in [1{\uparrow}k]$},
 there exists a unique   \mbox{$r_{n,i} \in R$}
such that \mbox{$s_{n,i} \in Gr_{n,i}$}, and we may choose some \mbox{$g_{n,i}  \in G$}
  such that $g_{n,i} r_{n,i} = s_{n,i}$.
We have a map
 \mbox{$r_{n,-}:[1{\uparrow}k] \to R, i\mapsto r_{n,i}$}.
Since $$\vert\{r_{n,-} \mid n \in [0{\uparrow}\infty[\,\,\}\vert \le \vert R \vert^k < \omega_0,$$
there exists some infinite subset $\mathbf{N}$ of
\mbox{$[0{\uparrow}\infty[$} and some map \mbox{$r_{-}:[1{\uparrow}k] \to R$}, $i \mapsto r_i$,
 such that, for each \mbox{$n \in \mathbf{N}$},  $r_{n,-} = r_-$,  and, hence,
\vspace{-1.8mm}  $$\textstyle S_n = \bigcup \limits_{i=1}^{k} g_{n,i}\{r_{n,i}\}
= \bigcup \limits_{i=1}^{k} g_{n,i}\{r_i\}.\vspace{-1.8mm}$$
 We have \eqref{eq:k}, and $k \in \mathbf{K}$.

For any $R$,  $S \in \mathbb{S} {-}\{\emptyset\}$, we let  $\mathbf{d}(R,S)$ denote the
length of the minimum\d1length $X$-paths that cross  both $R$ and $S$.
Set \mbox{$d := \mathbf{d}(R,S) \in [1{\uparrow}\infty[\,$}.
The  $X$-distance, in the usual sense,  from $R$ to $S$ equals   \mbox{$\max\{d{-}2,0\}$}.
 It may be seen that
 \mbox{$\P_{d}(R) \cap \P_{d}(S)$} is nonempty and
consists of the minimum\d1length  $X$-paths
with the properties that exactly one edge (the first or last) lies  in~$R $ and exactly one edge
(the last or first) lies in~$S$.
For each \mbox{$\ell \in [0{\uparrow}d[\,$},
$$\P_\ell(R \cup S) = \P_\ell(R) \vee P_\ell(S) \text{ and }
 \abs{\P_\ell(R \cup S)} = \abs{\P_\ell(R)} + \abs{\P_\ell(S)},$$
while $$\P_d(R \cup S) \subseteq  \P_d(R) \cup \P_d(S)\text{ and }
 \abs{\P_d(R \cup S)} < \abs{\P_d(R)} + \abs{\P_d(S)}.$$
(If $d = 1$, then  \mbox{$R \cap S \ne \emptyset$} and  $\P_d(R \cup S) = \P_d(R) \cup \P_d(S)$,
while if \mbox{$d \ge 2$},  then \mbox{$R \cap S = \emptyset$} and
$\P_d(R \cup S) \subset \P_d(R) \cup \P_d(S)$.)

Now  suppose that we have some  $k \in \mathbf{K}$ with~$k\ge 2$; we shall show that \mbox{$k{-}1 \in \mathbf{K}$}.
Here, we have~\eqref{eq:k}.  Consider any  \mbox{$n \in \mathbf{N}$}.  We set
$$d_{n} \coloneq \min \{\mathbf{d}(g_{n,i}R_i,g_{n,j} R_{j})
 \mid \text{$i,j \in [1{\uparrow}k]$ with $i < j$} \} \in [1{\uparrow}\infty[\,.\vspace{-.5mm}$$

We first prove that  \mbox{$\{d_N \mid N \in \mathbf{N}\}$} is finite.
For each $\ell \in [0{\uparrow}\infty[\,$,  set\vspace{-2mm}  $$c_\ell  \coloneq
\textstyle\sum\limits_{i=1}^{k}
   \vert\P_\ell(R_i) \vert =  \textstyle\sum\limits_{i=1}^{k}
   \vert\P_\ell( g_{n,i}R_i) \vert.\vspace{-3mm}$$ Then\vspace{-1mm}
$$    \vert\P_{d_n}(S_n)\vert =
  \vert\P_{d_n}(\textstyle\bigcup\limits_{i=1}^{k}  g_{n,i}R_i)\vert
   \le  \vert  \textstyle\bigcup\limits_{i=1}^{k} \P_{d_n}( g_{n,i} R_i) \vert
  <   \textstyle\sum\limits_{i=1}^{k}
   \vert\P_{d_n}( g_{n,i}R_i) \vert
   =  c_{d_n}.$$
For each $\ell \in [0{\uparrow} d_{n}[\,$,
$$
 \vert\P_{\ell}(S_n)\vert  = \vert\P_\ell(\textstyle\bigcup\limits_{i=1}^k g_{n,i}R_i)\vert
 = \vert \hskip-2pt\textstyle\bigvee\limits_{i=1}^{k}  \P_\ell(g_{n,i}R_i) \vert
 =  \textstyle\sum\limits_{i=1}^{k}
   \vert\P_\ell( g_{n,i}R_i) \vert
   = c_\ell.
$$
Hence,\vspace{-4mm}  $$ \textstyle \sum\limits_{\ell =0}^{d_n-1}  \hskip-2pt c_\ell  \ttt^\ell
 \sqsubseteq\, \SP(S_n) \,\sqsubset\, \sum\limits_{\ell =0}^{d_n}
  c_\ell  \ttt^\ell .$$
Now consider any $m \in \mathbf{N}$ such that  $m \ge n$.  Then
\mbox{$\SP(S_m) \,\sqsubseteq\,  \SP(S_n)$}, since $\SP(S_{-})$ is decreasing.
Hence, \vspace{-1mm}
$$ \textstyle \sum\limits_{\ell =0}^{d_m-1}  \hskip-2pt c_\ell  \ttt^\ell
\,\sqsubseteq\,\SP(S_m)\,\sqsubseteq\, \SP(S_n) \,\sqsubset\, \sum\limits_{\ell =0}^{d_n}
  c_\ell  \ttt^\ell .$$
Thus, $d_m {-}1 < d_n$.  Hence, $d_m \le d_n$.  It follows that \mbox{$\{d_N \mid N \in \mathbf{N}\}$} is finite,
and we may assume it has exactly one element, $d_\ast$,
by replacing $\mathbf{N}$ with a suitable infinite subset.

Fix
 \mbox{$i_n,j_n \in [1{\uparrow}k]$} such that $i_n < j_n$
and $$\mathbf{d}(g_{n,i_n} R_{i_n},g_{n,j_n} R_{j_n}) = d_n = d_\ast.$$
Now \mbox{$\bigl\{(i_N,j_N) : N \in \mathbf{N} \bigr\}$} is finite, and we may
assume it has exactly one element, $(i_\ast, j_\ast)$, by replacing $\mathbf{N}$ with a suitable infinite subset.
By renumbering the $R_i$, we may assume that $(i_\ast, j_\ast) = (1,k)$.  Now \vspace{-1mm}
$$d_\ast = \mathbf{d}(g_{n,1} R_{1},g_{n,k} R_{k}) = \mathbf{d}(R_{1},g_{n,1}^{-1}g_{n,k} R_{k}).$$

 Since $X$ is locally finite and $R_{1}$ and $R_{k}$ are finite sets of edges,
there exist only finitely many elements in
the $G$-orbit of $R_{k}$ whose $X$-distance  from~$R_{1}$ equals
 \mbox{$\max\{d_\ast{-}2,0\}$}.
Thus,
\mbox{$\{g_{N,1}^{-1}g_{N,k} R_{k} \mid N \in \mathbf{N}\}$}
is finite, and we may assume that it has exactly one element, $R_\ast$,
by replacing $\mathbf{N}$ with a suitable infinite subset.
Now
 \mbox{$ g_{n,1}^{-1}g_{n,k}R_{k} =  R_\ast$},  and then
\mbox{$g_{n,1}  R_{1} \cup g_{n,k} R_{k}  = g_{n,1} (R_{1} \cup R_\ast)$.}
Here, we may replace  $R_{1}$ with $R_{1} \cup R_\ast$
and $k$ with $k{-}1$,  and we see that $k{-}1 \in \mathbf{K}$.

This completes the proof of Theorem~\ref{th:berg3}.
\end{proof}

The  following is the locally finite case of a result of Bowditch and Dunwoody~\cite[$\S8$]{BHB}, which
was based on work of Bergman  \mbox{\cite[Lemma~1]{GMB1}} and Dunwoody and Swenson  \mbox{\cite[Lem\-ma~3.3]{DS}.}

\begin{theorem}\label{thm:nest} Let $X$ be any connected, locally finite, $G$-finite $G$-graph, and
  $\AA$~be any Boolean $G$-subalgebra of $\BB(X)$.
Then
$\irr(\AA)$ is a  $\comp$-stable,  nested  $G$-subset of~$\AA$
such that \mbox{$\emptyset \not\in \irr(\AA)$} and $ \gen{\irr(\AA)}_{{\BB}} =\AA$.
\end{theorem}

\begin{proof} It is clear that $\irr(\AA)$ is a $\comp$-stable  $G$-subset  of $\AA$
such that $\emptyset \not\in \irr(\AA)$.
By Theorem~\ref{th:berg3}, $\SP(\AA)$ is well-ordered, and then  a standard argument shows that
$\gen{\irr(\AA)}_{{\BB}} =\AA$.
It remains to show that $\irr(\AA)$ is nested.

Consider any $A'$, $B' \in\irr(\AA)$.
It suffices to show that $A'$ and $B'$ are nested.
Let us choose   $(A,B) \in \{A', A'^\comp\} \times \{B',B'^\comp\}$ to make $\SP(A \cap B)$
as $\sqsubset$-small as possible.  In particular, \mbox{$\SP(A \cap B) \sqsubseteq \SP(A \cap B^\comp)$},
and we see that
  $$A = (A \cap B) \cup (A \cap B^\comp) \in
\gen{ \{C \in \AA \mid \SP(C) \sqsubseteq \SP(A \cap B^\comp)\}}_{{\BB}}.$$
Since $A$  is  $\AA$-irreducible,  it is not the case that
 \mbox{$\SP(A \cap B^\comp) \,\sqsubset\,   \SP(A)$}.  Thus,
\begin{equation}\label{eq:symm2}
 \SP(A) \,\sqsubseteq\,   \SP(A \cap B^\comp).
\end{equation}
For any elements  \mbox{$C$, $D$} of $\BB(X)$, let us define \mbox{$\P(C,D)\coloneq \P(C) \cap \P(D^\comp)$} and
 \mbox{$\SP(C,D) \coloneq \Sigma(\P(C,D)) \in \integers[[\ttt]]$}.
If \mbox{$C \cap D = \emptyset$}, then \mbox{$\P(C,D)$} is  the set of all $X$-paths which begin in $C$ and end in $D$,
and, here, \mbox{$\SP(C,D) = \SP(D,C)$}.
Set
 \begin{align*}
&a\coloneq \SP(A\cap B, A^\comp\cap B), &&b\coloneq\SP(A\cap B, A^\comp\cap B^\comp),
&&c\coloneq \SP(A\cap B^\comp, A^\comp),
  \\& && &&d\coloneq  \SP(A\cap B^\comp, A \cap B).\\[-8mm]
\end{align*}
It is not difficult to see that\vspace{-2mm}
$$
 (a\,{+}\,b)\,{+}\,c   =   \SP(A, A^\comp)   =   \SP(A)
 \overset{\eqref{eq:symm2}}{\sqsubseteq}  \SP(A \cap B^\comp)
 = \SP(A \cap B^\comp, A^\comp \cup B)   = c\,{+}\,d.$$
 Hence, $b\sqsubseteq  d{-}a$.  By interchanging $A$ and $B$, we see that $b  \sqsubseteq  a{-}d$, also.
Thus,    $b  \sqsubseteq  0$.
Hence, there exists no $X$-path  which begins in $A \cap B$ and ends in $A^\comp  \cap B^\comp$;
 hence, $A \cap B$ or $A^\comp  \cap B^\comp$ is empty;
and, hence,  $A'$~and $B'$ are nested, as desired.
\end{proof}


\begin{remarks}  Throughout this section, we have considered
 connected, locally finite, $G$-finite $G$-graphs;
these include the Cayley graphs of finitely generated groups,
which are the graphs we shall be using.
 Both Bergman~\cite{GMB1} and Bowditch~\cite{BHB} consider
more general situations.
Bergman obtains similar results about connected, locally finite $G$-graphs.
 Bowditch obtains  results about countable groups.
We have not seen any way to use  these  generalizations for our narrow objective of
improving the proof of Theorem~\ref{thm:ast}.  In~\cite[II.2.20]{DD},
nested generating sets were constructed for Boolean algebras of arbitrary
connected graphs.  \hfill\qed
\end{remarks}

\section{Building trees from the Boolean algebra of a group}\label{sec:BB}

In this section we shall prove a substantial part of the finitely generable case of Theorem~\ref{thm:ast}.

Recall Definitions~\ref{defs3.4} and~\ref{defs3.5}.
 The following result  is implicit in the finitely generable case of the Almost Stability Theorem.
Dunwoody~\cite[4.7]{MJD} showed that
 \mbox{$G\substabs(\V(T))  \subseteq \Almosts(\FF)$}.

\begin{theorem} \label{thm:nested}  Suppose that  \mbox{$\rank(G) < \omega_{0}$}.
  For each $G$-finite $G$-subset $\FF$ of~$\BB(G)$,   there exists some   $G$-finite $G$-tree $T$
such that   \mbox{$G\substabs(\V(T)) = \Almosts(\FF)$} and  $\E(T)$ is $G$-quasifree.
\end{theorem}

\begin{proof}    Let $S$ be any finite generating set of $G$, and set $X \coloneq \X(G,S)$.
By Lemma~\ref{lem:equalB}, $X$ is a connected, locally finite, $G$-finite, $G$-free $G$-graph,
  \mbox{$\V(X)= G$}, and $\BB(X) = \BB(G)$.

Set $\AA \coloneq \gen{\FF}_{{\BB}}$ in $\BB(G) = \BB(X)$.

By Theorem~\ref{thm:nest},
$\gen{\irr(\AA)}_{{\BB}} = \AA$, and
 $\irr(\AA)$ is a $\comp$-stable, nested $G$-subset of $\AA$ such that
$\emptyset \not\in \irr(\AA) $.

Since $\FF$ is $G$-finite, there exists some $\comp$-stable,  $G$-finite  $G$-subset $\EE$ of
 $\irr(\AA)$ such that  $\FF \subseteq \gen{\EE}_{{\BB}}$,
and then   $\EE$  is nested, \mbox{$\emptyset \not\in \EE $},   and  $\gen{\EE}_{{\BB}} =\AA$.

We shall now see that $\EE$ is finitely separating.
Consider any edge $e$ of $X$.
For each $A \in \EE$, since $ G_e = \{1\}$,
there exist only finitely many $g \in G$ such that $ge \in \delta(A)$, or, equivalently, $e \in \delta(g^{-1}A)$.
Since $\EE $ is $G$-finite, we then see that there  exist only finitely many
$B \in \EE $ such that  $e\in \delta B$, that is,
 $(\iota e)^{\ast\ast} \cap \EE =_{\text a} (\tau e)^{\ast\ast} \cap \EE$.
Since $X$ is connected, it follows that $\EE$ is finitely separating.

Now $\EE$ is a  $\comp$-stable,  finitely separating, nested $G$-subset of~$ \PP(G) $
and \mbox{$\emptyset \not\in\EE$.}
By Theorem~\ref{thm:Dunwoody}\eqref{it:2}, $\T(\EE)$ is a
$G$-tree with edge $G$-set $\EE$.  The edge $G$-set of $X$ is $G$-free,
and, hence,   $\BB(X) - \{\emptyset,\V(X)\}$ is $G$-quasifree.
Hence,   $\EE$ is $G$-quasifree, that is,  $\E(\T(\EE))$  is $G$-quasifree.
Since $\EE$ is $G$-finite, we see that $\T(\EE)$ is $G$-finite.

It remains to show that   $G\substabs(\V(\T(\EE))) = \Almosts(\FF)$.
Since $$\gen{\EE}_{{\BB}} = \AA = \gen{\FF}_{{\BB}},$$
it is not difficult to show that  $\Almosts(\EE)=\Almosts(\AA) = \Almosts(\FF),$
and it suffices to show that $G\substabs(\V(\T(\EE))) = \Almosts(\EE)$.
Notice that if $\EE = \emptyset$, then $\T(\EE)$
is a single vertex stabilized by all subgroups of $G$,
in which case it is clear that \mbox{$ G\substabs(\V(\T(\EE))) = \Almosts(\EE)$}.
Thus, we may assume that $\EE \ne \emptyset$.

We shall use the following observations.
Consider any $e$, $f \in \EE$.
Notice that, for each \mbox{$g \in G$},  $$(gf \in 1^{\ast\ast}) \Leftrightarrow
 (1 \in gf)   \Leftrightarrow (g^{-1} \in f)
\Leftrightarrow \bigl(g \in f^{-1}\coloneq \{g^{-1} \mid g \in f\} \in \PP(G)\bigr).$$
With careful interpretation, we may write \mbox{$G\{f\} \cap 1^{\ast\ast} = f^{-1} \{f\}$}
in $\PP(\PP(G))$. By Theorem~\ref{thm:Dunwoody}\eqref{it:2},\eqref{it:3},
 $    \iota e   =_{\text a} 1^{\ast\ast} \cap\, \EE $
in $\PP(\PP(G))$, and, hence,
\mbox{$  G\{f\} \cap \iota e  =_{\text a} G\{f\} \cap 1^{\ast\ast}.$}
  We record
\begin{equation}\label{eq:eq2}
 f^{-1} \{f\}    =_{\text a} G\{f\} \cap \iota e   \text{ in } \PP(\PP(G)).
\end{equation}

Now consider any $H \in  G\substabs(\V(\T(\EE)))$.
Then there exists some $e \in \EE$ such that
  $H$ stabilizes $\iota e$ or $\tau e$.  Consider any $f \in \EE$. Notice that
  \eqref{eq:eq2} implies that
$$  f^{-1}   =_{\text a} \{g  \in G  \mid gf \in \iota e\}   \text{ in } \PP(G),$$ since these are
right $G_f$-sets and $ G_f$ is finite.  Hence,
\mbox{$f    =_{\text a}  \{g \in G \mid f \in g\iota e\}$,}
which is a right $G_{\iota e}$-set. Similarly,
\mbox{$f    =_{\text a}  \{g \in G \mid f \in g\iota (e^\comp)\}$} and, hence,
$$f  =_{\text a} \{g \in G \mid f \in \iota ge \text{ and } f \in \iota (ge^\comp)\},$$
which is a right $G_{\tau e}$-set.  Thus, $f$ is almost  a right $H$-set.
Thus, $H \in \Almosts(\EE)$.

For the converse, we now consider any $H \in \Almosts(\EE)$.
Consider any $e \in \EE$ and any   finite $G$-transversal $F$
in the $G$-finite $G$-set $\EE$.
For each $f \in F$,  we have \mbox{$f  \in  \EE$,} and, hence,
there exists some right $H$-subset  $A_f$ of~$G$ such that
\mbox{$ f  =_{\text a} A_f$} in $\PP(G)$.
We may then form the  $H$-set
$$w \coloneq \textstyle \bigcup\limits_{f \in F}  (A_f^{-1} \{f\})
=_{\text a} \bigcup\limits_{f \in F}  (f^{-1}\{f\})
 \overset{\eqref{eq:eq2}}{=_{\text a}}
 \bigcup\limits_{f \in F} (G\{f\}  \cap \iota e) = \iota e\text{ in } \PP(\PP(G)).$$
Set $d \coloneq \abs{(\iota e)\triangledown (w)} \in [0{\uparrow}\infty[\,$.
For each $h \in H$,   $$d = \abs{(h\iota e)\triangledown (hw)}
= \abs{(h\iota e)\triangledown (w)} \text{ and }
 \abs{(h\iota e)\triangledown (\iota e)} \le \abs{(h\iota e)\triangledown ( w)} +
 \abs{ (w)\triangledown(\iota e)} = 2d.$$
By Theorem~\ref{thm:Dunwoody}\eqref{it:2},
 the $\T(\EE)$-distance between   $ \iota e$ and $h\iota e$    is at most $2d$.
Hence, the  subtree of $\T(\EE)$ spanned by $H\iota e$ has finite diameter.
Consider any $H$-subtree $T$ of $\T(\EE)$ of minimum possible diameter.
Then  $T$ has at most one edge,
for, otherwise,  deleting from $T$ all vertices of valence one and the edges incident thereto
leaves an $H$-subtree of smaller diameter.  It follows  that $H$ stabilizes some
vertex  of     $\T(\EE)$.   Thus, $H \in  G\substabs(\V(\T(\EE)))$.
\end{proof}

\section{Digression 2:  Stallings' ends theorem}\label{sec:ends}

In this section, we shall deduce a form of Stallings' celebrated Ends Theorem~\cite[4.1]{Stallings2},
a result which inspired much subsequent work in combinatorial group theory,
including all the theory discussed in this article.

In the case where $G$ is finitely generable and $S$ is any finite generating set of~$G$,
let $\mathbb{S}$ denote the set of finite subsets of   $\E(\X(G,S))$,
and, for each  $E \in \mathbb{S}$,
let $\phi(E)$ denote the set of infinite components of  $\X(G,S){-}E$.
Then \mbox{$(\phi(E) \mid E \in \mathbb{S})$} forms an inverse directed system, and,
by a 1945 argument of Freudenthal~\cite[6.16.1]{Freudenthal},  the resulting inverse limit
is independent of the choice of finite generating set $S$. The elements of this inverse limit
are called the \textit{ends} of the group~$G$.

We wish to consider the graph-theoretical conditions
\begin{enumerate}[\hskip.9cm(a)]
\item  $G$ is finitely generable and has more than one end.
\item  There exists some $G$-tree such that the edge $G$-set is $G$-quasifree and
no vertex is $G$-stable.
\end{enumerate}
and the   group-theoretical conditions
\begin{enumerate}[\hskip .9cm(a$'$)]
\item  $\BB(G)$ has some element $A$ such that both $A$ and $G{-}A$ are infinite.
\item   Either $G$ is countably infinite and locally finite, or there exists some finite subgroup $B$ of $G$ such that $G$
 is  a free product with amalgamation $C \ast_BD$ where $B < C$ and $B < D$
or $G$ is an HNN extension $C \ast_B \phi$ where $B \le C$ and $\phi:B \to C$ is a monomorphism.
\end{enumerate}

In 1949, Specker~\cite{Specker} showed that if $G$ is finitely generable, then (a)$\Leftrightarrow$(a$'$).
 Subsequently, it became a common practice to use some cohomological form of (a$'$) as a definition for
 `$G$ has more than one end'  even if $G$ is not finitely generable and, hence,
  ends of $G$ are not defined.

By Bass-Serre theory, (b)$\Leftrightarrow$(b$'$); see~\cite{Serre},~\cite[I.4.12]{DD}.

It is not difficult to show that (b)+(b$'$)$\Rightarrow$(a$'$); see~\cite[IV.6.10]{DD}.

In 1970, Stallings~\cite[4.1]{Stallings2}
proved  that  (a)$\Rightarrow$(b$'$); notice that
no finitely generable group is both infinite and locally finite.
He remarked that a communication from Dunwoody inspired his short proof of his key lemma~\cite[1.5]{Stallings2}
(which is actually the Cayley-graph case of a result of Bergman~\cite[Theorem~1]{GMB1}).
In 1979, Dunwoody~\cite[4.4]{MJD} proved directly that if $G$ is  finitely generable, then (a$'$)$\Rightarrow$(b);
this is the restatement of Stallings' result in which the graph-theoretic hypothesis (a)
is  replaced with the group-theoretic condition~(a$'$) and
 the group\d1theoretic conclusion (b$'$) is  replaced with the  graph-theoretic condition~(b).

\begin{stallings}\label{th:ends} If $G$ is finitely generable and there exists some element~$A$ in \mbox{$\BB(G)$}
such that $A$ and $G{-}A$ are infinite, then
  there exists some $G$-tree  such that no vertex is $G$-stable,
the edge $G$-set is $G$-quasifree, and the number of $G$-orbits  of edges equals $1$.
\end{stallings}

\begin{proof}
Set \mbox{$\FF \coloneq \{gA \mid g \in G\} \subseteq \BB(G)$.}
By Theorem~\ref{thm:nested},  there exists
some $G$-finite $G$-tree $T$ such that
  \mbox{$G\substabs(\V(T)) = \Almosts(\FF)$} and $\E(T)$ is $G$-quasifree.   Notice that
$G \not\in \Almosts(\{A\}) = \Almosts(\FF) = G\substabs(\V(T))$.
Now we collapse $G$-orbits of edges of $T$, one $G$-orbit at a time.  At some first stage,
a   $G$-stable vertex appears, and then the $G$-orbit of edges that has just been collapsed is the edge
$G$-set of
a $G$-tree which has the desired properties.
\end{proof}

\begin{remarks}
In 1968,  Stallings~\cite{Stallings1}
had proved a special case of (a)$\Rightarrow$(b$'$), and had written the following:
 ``Since ``ends" are, after all, a topological
kind of thing, there is no need to make a profuse apology for a topological kind of proof.
However, maybe there is some algebraic translation of this
which will go over to infinitely generated groups."
An algebraic translation which went over to all groups
was given in 1989  when Dicks and Dunwoody proved  in~\cite[IV.6.10]{DD} that (a$'$)$\Rightarrow$(b$'$). An
important advance in the theory had been made by Holt in 1981, who showed  in~\cite{Holt}
 that if  $G$ is locally finite, then (a$'$)$\Rightarrow$(b$'$); notice  that no locally
finite group is  an HNN extension or a proper free product with amalgamation.\qed
\end{remarks}

\section{The finitely generable case of the  Almost Stability Theorem}\label{sec:6}

Recall Definitions~\ref{defs:6.1}.
We may now prove the case of the Almost Stability Theorem~\ref{thm:ast} where $G$ is finitely generable.

\begin{theorem}\label{thm:astfg} Suppose that $\rank(G) < \omega_{0}$.
If  $E$ and $Z$ are any $G$-sets such that
$E$ is $G$-quasifree and  each element's $G$-sta\-bi\-li\-zer  stabilizes  some element of~$Z$,
then each $G$-sta\-ble almost equality class in the $G$-set  $\Map(E,Z)$
is the vertex $G$-set of some $G$-tree.
\end{theorem}

\begin{proof}  Let $V$ be any $G$-sta\-ble almost equality class in   $\Map(E,Z)$.
We shall prove a sequence of three equalities which will relate $V$ to a $G$-incompressible $G$-tree.

Let $S$ be any finite generating set of $G$, and set \mbox{$X \coloneq \X(G,S)$}.
Then $X$ is a connected, locally finite, $G$-finite, $G$-free $G$-graph,
\mbox{$\V(X)= G$}, and $\BB(X) = \BB(G)$.

As $V$ is nonempty, we may choose an element $v$ of $V$.

Consider any \mbox{$e \in E$}.  We then have a map \mbox{$\langle _{-}v,e\rangle \colon G \to Z$},
\mbox{$g \mapsto \langle gv,e\rangle$}, and we shall be interested in the set of
fibres thereof,  \mbox{$\{\langle _{-}v,e\rangle^{-1}(\{z\}) \mid z \in Z\}.$}
The set of edges of $X$ which are \textit{broken} by this same map \mbox{$\langle _{-}v,e\rangle \colon \V(X) \to Z$} is
$$
\delta(\langle _{-}v,e\rangle)  \coloneq \{ (g,s) \in \E(X) \mid   \langle \iota_{X}(g,s)v,e \rangle   \ne
  \langle \tau_{X}(g,s)v,e \rangle\}.$$
Thus, $$\bigl((g,s) \in \delta(\langle _{-}v,e\rangle)\bigr)  \Leftrightarrow   \bigl(\langle gv ,e \rangle   \ne
  \langle gsv,e \rangle\bigr) {\Leftrightarrow} \bigl(( e \in (gv) \triangledown (gsv) = g(v\triangledown(sv))\bigr).$$
 Hence, \vspace{-2mm}
$$\delta(\langle _{-}v,e\rangle)    =   \{ (g,s) \in \E(X) \mid  g^{-1}e \in v \triangledown(sv) \}.$$
For each $s \in S$,
$v \triangledown (sv)$ is  finite, since
$v =_{\text a} sv$ in $\Map(E,Z)$. Since
 $ G_e $  and $S$ are finite,  we see that $\delta(\langle _{-}v,e\rangle)$ is finite.
Since $X$ is connected,  the set of
fibres of~$\langle _{-}v,e\rangle$  is finite, and each fibre of~$\langle _{-}v,e\rangle$
is then an element  of~$\BB(X)=\BB(G)$.

Set $ E_G  \coloneq  \{e \in E : \langle _{-}v,e\rangle \text{ is not constant} \}$.
Then $$
E_G  = \{e \in E : \delta(\langle _{-}v,e\rangle)  \ne \emptyset\}
=   \textstyle\bigcup\limits_{g\in G}\bigcup\limits_{s\in S} (g(v \triangledown(sv))).
$$
Since  $\bigcup\limits_{s\in S} (v \triangledown(sv) ) $ is finite, we see that $E_G$  is $G$-finite.

For $g \in G$, $e \in E$, $z \in Z$, we have \mbox{$g(\langle _{-}v,e\rangle ^{-1}(\{z\}))
= \langle _{-}v,ge\rangle ^{-1}(\{gz\})$.}
Set
$\FF \coloneq \{ \langle _{-}v,e\rangle^{-1}(\{z\}) \mid z \in Z, e \in E\}$. Then
$\FF$ is a $G$-finite $G$-subset of~$\BB(G)$.

We shall now prove
\begin{equation}\label{eq:2aa}
G\substabs(V) \subseteq \Almosts(\FF),
\end{equation}
that is, for each \mbox{$H \in G\substabs(V)$}, each
 \mbox{$\langle _{-}v,e\rangle^{-1}(\{z\}) \in \FF$} is almost equal to  some right $H$-set.

\begin{proof}[Proof of~\eqref{eq:2aa}]
Here, $e \in E$, $z \in Z$, and $H$  stabilizes some element $w$ of $V$.
Since \mbox{$w =_{\text a} v$} and   $ G_e $ is finite,
we see that, for all but finitely many $ g \in G$ , we have \mbox{$
\langle  v,g^{-1}e\rangle  = \langle w ,g^{-1}e \rangle$,} that is,
\mbox{$
g^{-1}\langle  gv, e\rangle  = g^{-1}\langle gw ,e \rangle$,} that is,
\mbox{$
 \langle  gv, e\rangle  =  \langle gw ,e \rangle$.}
Thus \mbox{$
\langle _{-}v,e\rangle   =_{\text a} \langle _{-}w,e\rangle
$}.  In particular, \mbox{$
\langle _{-}v,e\rangle^{-1}(\{z\})  =_{\text a} \langle _{-}w,e\rangle^{-1}(\{z\})
$}, and the  latter set is easily seen to be a right $H$-set.
This completes the proof of~\eqref{eq:2aa}.
\end{proof}

We shall next prove
\begin{equation}\label{eq:2ab}
G\substabs(V) \supseteq \Almosts(\FF).
\end{equation}

\begin{proof}[Proof of~\eqref{eq:2ab}] Consider any \mbox{$H \in \Almosts(\FF)$}.
It suffices to construct some\linebreak  \mbox{$w \in \Map(E,Z)$} such that \mbox{$w =_{\text a} v$} and
$H$ stabilizes $w$.

Set $E_H \coloneq \{e \in E : \langle _{-}v,e\rangle\vert_H     \text{ is  not  constant}\}$,
an $H$-subset of $E_G$.  Thus, for \mbox{$h \in H$} and \mbox{$e \in E-E_H$,} we have
\mbox{$  \langle hv,e\rangle  =    \langle v,e\rangle  $,} and we see that
 $H$ stabilizes $\langle v, _{-}\rangle \vert_{E-E_H}$.

Consider any $e \in E_G$.  We  saw above that
$\langle _{-}v,e\rangle$ takes only finitely many values in~$Z$,
and we are assuming that, for each $z \in Z$,
$\langle _{-}v,e\rangle^{-1}(\{z\})$ is almost equal to a right $H$-subset of $G$.  Hence,
$\langle _{-}v,e\rangle\vert_H$ is
almost equal to a  constant map, and, also,
 for all but  finitely many $g$ in a right $H$-transversal in $G$, $\langle _{-}v,e\rangle\vert_{gH}$ is
 constant;  here,  $\langle g_{-}v,e\rangle\vert_{H}$ is
 constant,  $g^{-1}\langle g_{-}v,e\rangle\vert_{H}$ is
  constant, $\langle _{-}v,g^{-1}e\rangle\vert_{H}$ is
  constant, $g^{-1}e \in E{-}E_H$,  and $Hg^{-1} e \cap  E_H = \emptyset$.
It follows that $Ge \cap  E_H $ is $H$-finite.

We also saw above that $E_G$  is $G$-finite.
It now follows that $E_H$ is $H$-finite.

Let us  deal first with the case where  $H$ is infinite.  For each $e \in E$,
as $\langle _{-}v,e\rangle\vert_H$ is  almost
equal to a constant map and $H$ is infinite, there exists a unique $z_e \in Z$ such that,
for all but finitely many $h \in H$,  $\langle hv,e\rangle = z_e$.
For each $e \in E$ and \mbox{$h_0 \in H$}, we see that, for all but finitely many \mbox{$h  \in H$},
 $\langle h_0 h v, e\rangle = z_e$, and then
$$ \langle  h v, h_0^{-1} e\rangle = h_0^{-1}\langle h_0hv, e\rangle    = h_0^{-1} z_e;$$
thus, $z_{h_0^{-1}e} = h_0^{-1} z_e$.
Set \mbox{$w: E \to Z$, $e \mapsto \langle w, e \rangle := z_e$}; then $H$ stabilizes~$w$, since
  $$\langle h_0w, e \rangle = h_0\langle w, h_0^{-1}e \rangle = h_0z_{h_0^{-1} e} = z_e = \langle w, e \rangle.$$
For each \mbox{$e \in E{-}E_H$}, \mbox{$\langle v, e\rangle   = z_e = \langle w, e\rangle$};
thus,  \mbox{$\langle v, _{-}\rangle\vert_{E-E_H} =  \langle w, _{-}\rangle \vert_{E- E_H}$}.
 For each \mbox{$e \in E_H$},
for   all but finitely many \mbox{$h \in H$}, \mbox{$\langle hv, e\rangle  = z_e = \langle w, e\rangle$,} and here
$$\langle v, h^{-1}e\rangle  = h^{-1} \langle hv, e\rangle   = h^{-1} \langle w, e\rangle
= \langle  h^{-1}  w,  h^{-1}  e\rangle = \langle   w,  h^{-1}  e\rangle.$$
Since $E_H$ is $H$-finite, we see that $\langle v, _{-}\rangle\vert_{E} =_{\text a}  \langle w, _{-}\rangle \vert_{E}$.
Hence, \mbox{$w \in V$} and \mbox{$H \in G\substabs(V)$}.

It remains to deal with the case where
 $H$ is finite.  Here, the $H$-finite set $E_H$ is finite.
Since \mbox{$G\substabs(E) \subseteq  G\substabs(Z)$} by hypothesis, there exists
 some $G$-stable \mbox{$u \in \Map(E,Z)$.} Define
 $w$ to be the element of $\Map(E,Z)$ such that
$\langle w, _{-}\rangle\vert_{E-E_H} = \langle v, _{-}\rangle\vert_{E-E_H}$ and
\mbox{$\langle w, _{-}\rangle\vert_{E_H}= \langle u, _{-}\rangle\vert_{E_H}$.}
Since $H$ stabilizes  both $\langle v, _{-}\rangle\vert_{E-E_H} $  and
 \mbox{$\langle u, _{-}\rangle\vert_{E_H}$,} we see that $H$ stabilizes $w$.
 Since  $E_H$ is  finite,
 \mbox{$\langle w, _{-}\rangle\vert_{E} =_{\text a}  \langle v, _{-}\rangle \vert_{E}$.}
Hence $w \in V$ and  $H \in G\substabs(V)$.

This completes the proof of~\eqref{eq:2ab}.
\end{proof}

By combining~~\eqref{eq:2aa} and ~\eqref{eq:2ab}, we find that
\begin{equation}\label{eq:2a}
G\substabs(V) = \Almosts(\FF).
\end{equation}

By  Theorem~\ref{thm:nested}, since $\FF$ is a $G$-finite $G$-subset of $\BB(G)$,
 there exists some $G$-finite $G$-tree $T_1$
 such that $\E(T_1)$ is $G$-quasifree and
\begin{equation}\label{eq:4a}
\Almosts(\FF) = G\substabs(\V(T_1)).
\end{equation}

By
successively collapsing   $G$-orbits of any $G$-com\-pressible edges of $ T_1$,
we arrive at a $G$-incompressible
$G$-tree $T_2$ such that
\begin{equation}\label{eq:6a}
G\substabs(\V(T_1)) = G\substabs(\V(T_2)).
\end{equation}

 In summary,
$$
G\substabs(V)   \hskip-2pt  \overset{\eqref{eq:2a}}{=} \hskip-2pt  \Almosts(\FF)
\hskip-2pt  \overset{\eqref{eq:4a}}{=}\hskip-2pt
G\substabs(\V(T_1)) \hskip-2pt  \overset{\eqref{eq:6a}}{=}\hskip-2pt   G\substabs(\V(T_2)).
$$

 As $ G\substabs(V)  = G\substabs(\V(T_2))$,
there exist  $G$-maps $\phi \colon V {\to} \V(T_2)$ and  \mbox{$\psi \colon  \V(T_2) {\to}   V$.}
Since $\V(T_2)$ is $G$-incompressible,
the $G$-map \mbox{$\phi \circ \psi \colon \V(T_2) \to \V(T_2)$}   must be bijective.
Hence $\psi$ is injective, and we may identify
 $\V(T_2)$ with a $G$-subset of~$V$, and $T_2$ with a $G$-subtree of the $G$-graph $\complete(V)$.
The $G$-subgraph $T_3$ of $\complete(V)$ with vertex $G$-set $V$ and edge $G$-set
$$\E(T_3):= \E(T_2) \cup \{(v, \phi(v)) \mid  v \in   V{-}\V(T_2)\}$$
is  a maximal subtree of   $\complete(V)$,
as desired.
\end{proof}

\section{Preliminary results about trees}\label{sec:8}

In the remainder of this article we shall describe some simplifications which may be made in the proof
of the general case of the Almost Stability Theorem.  We will not
 simplify the proofs of the preliminary results about trees.
We collect together the statements of these here, for the convenience of the reader.
The proofs currently known are rather technical and will not be given  here.

\begin{definitions}  Let $\overline T = (\overline V, \overline E, \overline \iota, \overline \tau)$ be
any $G$-tree such that $\overline E$ is $G$-quasifree.
Let $F$ be any $G$-forest with $G$-quasifree edge $G$-set such that
the $G$-set of components of $F$ is $\overline V$.  Thus,
 $F = \bigvee\limits_{w \in  {\overline V}}  T_w$,   for each   $w \in \overline V$,
$T_w$ is a $G_w$-tree with $G_w$-quasifree edge $G_w$-set, and,  for each  $g \in G$,  $g(T_w) = T_{gw}$.

We shall now extend $F$ to a $G$-graph $F \vee \overline E$ by adding $\overline E$ to the
edge $G$-set of $F$ and extending the incidence maps $\iota$ and $\tau$  to
$\overline E$  as follows.  Let $S$ be any $G$-transversal in $\overline E$.
Consider any $\overline e \in S$.  Then
$G_{\overline e}$ is a finite subgroup of $G_{\overline \iota(\overline e)}$,
and, hence, $G_{\overline e}$ stabilizes
some vertex of the $G_{\overline \iota(\overline e)}$-tree $T_{\overline \iota (\overline e)}$.
We take some $G_{\overline{e}}$\,-stable
 vertex of $G_{\overline \iota (\overline e)}$ to be $\iota \overline e$.
For each $g \in G$, we define $\iota(g\overline e) \coloneq g(\iota (\overline e))$, which  is well-defined.
This defines $\iota \colon \overline E \to \V(F)$.  We define $\tau\colon \overline E \to \V(F)$
in a similar manner.  This completes the definition of $F \vee \overline E$.

Collapsing the edges of the subforest $F$ in $F \vee \overline E$ leaves the tree $\overline T$. It follows that
$F \vee \overline E$ is a $G$-tree with $G$-quasifree edge $G$-set.  We say that
$F \vee \overline E$ is a $G$-tree \textit{obtained from $\overline T$ by $G$-equivariantly  blowing up
each $w \in \overline V$ to $T_w$}.
\hfill\qed
\end{definitions}

\begin{definitions}  Let $T$ be any $G$-finite $G$-tree with $G$-quasifree edge $G$-set.
For each $n \in [1{\uparrow}\infty[$\,, set $\E_n(T) \coloneq \{e \in \E(T) : \vert G_e \vert = n \}$, and set
  $$\textstyle \size (T) \coloneq
 \vert G \leftmod \!\E(T) \vert -  \vert G \leftmod\! \V(T) \vert \, + \hskip-8pt \sum\limits_{n\in [1{\uparrow}\infty[}
\hskip -3pt \vert G \leftmod \!\E_n(T) \vert \ttt^n \,\,\in \,\, \integers[ \ttt ].$$
\vskip-.8cm \hfill\qed \vskip.3cm
\end{definitions}

\begin{lemma}\label{lem:fg} Let $T$ be any $G$-tree with $G$-quasifree edge $G$-set,
 $w$ be any vertex of~$T$, and  $H$ be any subgroup of $G_w$.
If $\rank(G \text{ {\normalfont rel} } H) < \omega_{0}$, then the following hold.
\begin{enumerate}[\normalfont (i)]
\item $\rank (G_w \text{ {\normalfont rel} } H) < \omega_{0}$.
\item For each $v \in \V(T) - Gw$,   $\rank(G_v) < \omega_{0}$.
\item There exists some $G$-finite $G$-incompressible $G$-tree $\overline T$
such $\E(\overline T)$ is $G$-quasi\-free and $G\substabs(\V(\overline T)) = G\substabs(\V(T))$.
\end{enumerate}
\end{lemma}

\begin{proof} (i) and (ii) hold by~\cite[III.8.1]{DD}, for example.

(iii). Let $S$ be any finite subset of $G$ such that $1 \in S$ and $H \cup S$ generates $G$.
Let $T_0$ be any finite subtree of $T$ containing $Sw$.  Then $G(T_0)$ is a $G$-finite
$G$-subforest of~$T$.  Moreover,   $G(T_0)/\E(G(T_0))$
consists of a single  $G$-orbit in which the image of $w$ is stabilized by $H \cup S$; that is, $G(T_0)$ has only one
component.  Thus,  $G(T_0)$ is a $G$-finite $G$-subtree of $T$.

For each $v \in \V(T)$, $G_v$ stabilizes both  $v$ and $\{G(T_0)\}$.  It follows that $G_v$ stabilizes the (unique)
vertex of $G(T_0)$ which is closest to $v$.  Hence
$$G\substabs(\V(G(T_0))) = G\substabs(\V(T)).$$ Now successively collapsing $G$-orbits of $G$-compress\-ible
edges in $G(T_0)$ leaves a $G$-finite $G$-incompressible $G$-tree $\overline T$
such that $\E(\overline T)$ is $G$-quasifree and
$$G\substabs(\V(\overline T)) = G\substabs(\V(G(T_0))) = G\substabs(\V(T));$$
see~\cite[III.7.2]{DD}.
\end{proof}

\begin{lemma}\label{lem:sizes}  Let $T_1$ and $T_2$ be any $G$-finite, $G$-incompressible
$G$-trees with $G$-quasifree edge $G$-sets.
If   $G\substabs(\V(T_2)) \subseteq G\substabs(\V(T_1))$, then the  following hold.
\begin{enumerate}[\normalfont (i)]
\item\label{it:1size} $3\vert  G \leftmod T_2 \vert \le \vert G \leftmod T_1  \vert.$
\item\label{it:3size} $\size (T_2) \sqsubseteq \size (T_1)$ in $\integers[ \ttt ]$.
\item\label{it:2size} If $\size (T_2) = \size (T_1)$, then
$G\substabs(\V(T_2)) = G\substabs(\V(T_1))$.
\end{enumerate}
\end{lemma}

\begin{proof} By~\cite[III.7.5]{DD}, \mbox{$ \vert G \leftmod\!\E(T_1) \vert +  \vert G \leftmod\!\V(T_1) \vert \ge
  \vert G \leftmod\!\V(T_2) \vert$} and (ii) and (iii) hold.  By (ii),
 \mbox{$ \vert G \leftmod\!\E(T_1) \vert -  \vert G \leftmod\!\V(T_1) \vert \ge
\vert G \leftmod\!\E(T_2) \vert -  \vert G \leftmod\!\V(T_2) \vert$.}
By multiplying the former inequality by 2 and adding the result to the latter inequality, we see that (i) holds.
\end{proof}

\section{Notation used in the proof of the general case}\label{sec:9}

Throughout the remainder of the article, the following will apply.

\begin{notation}\label{not:gen}   Let $E$ and $Z$  be  any $G$-sets, and
 $ V$ be any  $G$-stable almost equality class in the $G$-set  $\Map(E,Z)$.

Suppose that  there exists some $G$-stable element in $\Map(E,Z)$
and that $E$ is $G$-quasifree.

By Remarks~\ref{rems:1}(i),  the connected $G$-graph $\complete(V)$ has $G$-quasifree edge $G$-set.

Here $V$ is nonempty.  Let us fix $v_0 \in V$.

For any subset $E'$ of $E$,  we have $E = E' \vee (E{-}E')$, and we identify
$$\Map(E,Z) = \Map(E',Z) \times \Map(E {-}E',Z) = \hskip-10pt
 \textstyle\bigvee\limits_{w \in \Map(E {-}E',Z)} \hskip-3pt (\Map(E',Z) \times  \{w\}),$$
where, for each  $w \in \Map(E {-}E',Z)$, we write $\Map(E',Z) \times  \{w\}$ for the fibre over $w$ of the
restriction map   $$\Map(E,Z) \to \Map(E{-}E',Z), \quad v \mapsto \langle v, _{-} \rangle  \vert_{{E-E'}}.$$
For each \mbox{$v \in \Map(E,Z)$,} we identify $$\{v\} = \{\langle v, _{-} \rangle  \vert_{{ E'}}\} \times
\{\langle v, _{-} \rangle  \vert_{{E-E'}}\},$$ and by abuse of notation we shall write
$$ v  =  \langle v, _{-} \rangle  \vert_{{ E'}}  \times
 \langle v, _{-} \rangle  \vert_{{E-E'}}.$$
We denote by $\pi_{{E'}}$ the self-map of $\Map(E,Z)$ defined by
$$\pi_{{E'}}(v) \coloneq  \langle v, _{-} \rangle  \vert_{{ E'}}
\times \langle v_0, _{-} \rangle  \vert_{{E-E'}} ;$$ thus, the image of
$\pi_{{E'}}$  equals the fibre over $\langle v_0, _{-} \rangle  \vert_{{E-E'}}$.

We denote by $\overline \V(E')$ the image of $V$ under the restriction/projection map
$\Map(E,Z) \to \Map(E',Z)$. Then $\overline \V(E')$ is a $G_{{E'}}$-stable almost equality class
in $\Map(E',Z)$.  Similarly, we also have a $G_{{E'}}$-stable almost equality class
$\overline \V(E{-}E')$ in $\Map(E{-}E',Z)$, and we have the identifications
 $$ V  = \overline \V(E') \times \overline \V(E{-}E') =
\textstyle\bigvee\limits_{w \in \overline \V(E{-}E')}(\overline \V(E') \times \{w\}).$$
We may construct $V$ as $G_{{E'}}$-set
by  blowing up  each $w \in  \overline \V(E{-}E')$ to the $(G_{ E'})_w$-set
$\overline \V(E') \times \{w\} \subseteq V$.
The restriction map
$\Map(E,Z) \to \Map(E',Z)$ carries
\mbox{$\overline \V(E') \times \{w\}$} bijectively to $\overline V(E')$
respecting the $(G_{E'})_{w}$-action.   For some purposes,
 we shall be able to identify $\overline \V(E') \times \{w\}$ with
the  $(G_{ E'})_w$-stable almost equality class $\overline \V(E')$ in $\Map(E',Z)$.

Set $w_0 \coloneq \langle v_0, _{-} \rangle  \vert_{{E-E'}}\in \overline V(E-E')$. We write
$$
\V(E') \coloneq  \overline \V(E') \times \{ w_0 \} = \{v \in V :   \langle v, _{-} \rangle
  \vert_{{E-E'}} =   \langle v_0, _{-} \rangle  \vert_{{E-E'}} \},$$
the fibre over $w_0$.
Then $v_0 \in \V(E') \subseteq V$.  Also,
$\pi_{{E'}}$ maps $V$ to  $\V(E')$ fixing each element of $\V(E')$ and
respecting the $(G_{E'})_{w_0}$-action.

Consider any subgroup $H$  of $G$.

We define $E_H \coloneq   \{e \in E :
\langle _{-}v_0,e \rangle\vert_{H}  \text{ is not constant} \}$, an $H$-subset of $E$.
Notice that $H$ stabilizes  $\langle v_0,_{-}\rangle \vert_{{E-E_H}}$, and $E_H$ is the smallest
 $H$-subset of $E$ with this property.
Also, $$\V(E_H) =  \{ v \in V :
  \langle v, _{-} \rangle  \vert_{E-E_H} =    \langle v_0, _{-} \rangle  \vert_{E-E_H} \}
= \overline V(E_H) \times \{\langle v_0, _{-} \rangle  \vert_{E-E_H}\},$$
and $\V(E_H)$ is an $H$-subset of $V$  that is isomorphic to the $H$-stable almost equality class
$\overline \V(E_H)$ in $\Map(E_H,Z)$.

We wish to  show that some maximal subtree of $\complete(V)$ is  $G$-stable.
It suffices to
show there exists some $G$-subtree $T_{G}$ of $\complete(V)$ with vertex $G$-set $\V(E_{G})$,
for then  $V$ itself is the vertex $G$-set of the $G$-subtree of
$\complete(V)$ with  edge $G$-set
$\E(T_{G}) \cup \{ (v,\pi_{{E_G}}(v)) \mid v \in V{-}\V(E_{G})\}$.

Let $\G$ denote the class   consisting of all those groups   for which the
Almost Stability Theorem~\ref{thm:ast} holds.
\hfill\qed
\end{notation}

To show that $G \in \G$, we may assume that
 {\normalfont Notation~\ref{not:gen} } holds, and it suffices to
show there exists some $G$-subtree  of $\complete(V)$ with vertex $G$-set $\V(E_{G})$.

\section{Finitely generable extensions}\label{sec:10}

The following  result is a modified version of~\cite[III.7.6]{DD}
differing mainly in the additional hypothesis that
\mbox{$G \in \G$.}
The important points are that this weaker form now suffices for our purposes  and the proof
is  simplified in two places by the additional assumption.

\begin{theorem}\label{thm:first} Let {\normalfont Notation~\ref{not:gen}} hold, and
suppose that the following hold:  \mbox{$G \in \G$};
\mbox{$\rank (G \text{\normalfont { rel }} H) < \omega_{0}$;}
for each \mbox{$g \in G{-}H$,} $gE_{H} \cap E_{H} = \emptyset$;
 and,  there exists
 some  $H$-subtree $T_{H}$ of $\complete(V)$   with vertex $H$-set $\V(E_{H})$\,.
Then there exists some $G$-subtree $T_{G}$ of $\complete(V)$ with vertex $G$-set $\V(E_{G})$
such that $T_{H} \subseteq T_{G}$.
\end{theorem}

\begin{proof}
Let $S$ be a finite subset of $G$ such that $H \cup S$ generates $G$.

Set $V_{{\infty}} \coloneq \{v \in \V(E_{G}) - G(\V(E_H)):  G_{{v}} \text{ is infinite}\}$.
For the moment, let $W$ be any  finite  subset of $V_{{\infty}}$.
In~\eqref{eq:infty}, we shall see that $V_{{\infty}}$ is $G$-finite, and   then take
$W$ to be a $G$-transversal in $V_{{\infty}}$.

 It is not difficult to show that, for each $g \in G{-}H$,  $\pi_{{E_H}}$
 sends every element  of $g\V(E_{H})$ to the single point $\langle gv_0,_{-} \rangle \vert_{E_H}
\times \langle v_0,_{-} \rangle \vert_{E-E_H}$.
Recall that  $\pi_{{E_H}}$ fixes each element of~$\V(E_{H})$.  We may then use
the set map $\pi_{{E_H}}$ to construct a graph map $G(T_H) \to T_H$ which collapses each edge in
$(G{-}H)(T_H)$ and acts as the identity map on~$T_H$.   Thus, whenever two vertices of $T_H$
are joined by a $(G{-}H)(T_H)$-path, the two vertices must be equal.
It then follows that  $G(T_{H})$ is a $G$-subforest of $\complete(V)$.
Set \mbox{$Y \coloneq   G(W) \vee G(T_{H})$,} also a $G$-subforest of   $\complete(V)$.

For any subset $E'$ of $E$,  we have the restriction map $\V(Y) \to \Map(E',Z)$,
$v  \mapsto  \langle v,_{-} \rangle \vert_{{\mkern-2mu E'}} $. \vspace{.5mm}

  The map \mbox{$\V(Y) \to \PP(\PP(\V(Y)))$,
$v \mapsto v^{\ast\ast} \coloneq \{ \varepsilon \in \PP(\V(Y)) \mid v \in \varepsilon\}$,}
will be identified with the map
\mbox{$\V(Y) \to \Map(\PP(\V(Y)),\integers_2)$},
$v \mapsto \langle v,_{-} \rangle \vert_{{\PP(\V(Y))}}$,
where, for each \mbox{$(v,\varepsilon ) \in \V(Y) \times \PP(\V(Y))$},
we set $$\langle v, \varepsilon  \rangle \coloneq \begin{cases}
1 \in \integers_2  &\text{if $v \in \varepsilon$,}\\
 0 \in \integers_2 &\text{if $v \in \varepsilon^\comp \coloneq \V(Y){-}\varepsilon $.}
\end{cases}$$
For each subset $\mathcal{E}$ of $\PP(\V(Y))$, we have
the restriction map   $\V(Y) \to \Map(\mathcal{E}, \integers_2)$,
$v \mapsto \langle v,_{-} \rangle \vert_{{\mathcal{E}}} $.

For each $e \in \E(T_H)$, we set
\begin{equation}\label{eq:dd}
  e^{\ast\ast} \hskip-1.5pt\coloneq  \hskip-1.5pt \{v {\in} \V(Y)  \hskip-1.5pt \mid \hskip-1.5pt
 \text{the reduced   $T_{H}$-path  from
$ \pi_{{E_H}}\!(v)$    to $\tau_{{T_H}}\!(e)$ crosses $e$} \}.
\end{equation}
Here, $\iota_{{T_H}}(e) \in e^{\ast\ast}$.
For each $g \in G{-}H$, $g\V(E_{H})$ is mapped to a single point in $\V(E_{H})$
by $ \pi_{{E_H}}$. It follows that
$\delta_{Y}(e^{\ast\ast})  = \{e\}$.
For each $h \in H$, we have $(he)^{\ast\ast} =  h(e^{\ast\ast})$.
For each $g \in G{-}H$, we  write $(ge)^{\ast\ast} \coloneq g(e^{\ast\ast})$, and this is well-defined.
For each subgraph $Y'$ of $Y$, we define $\EE(Y') \coloneq \{e^{\ast\ast} \cap \V(Y') \mid e \in \E(Y')\}$.
For each  $e \in \E(Y)$,
 $\delta_{Y}(e^{\ast\ast})  = \{e\}$. It follows that
$\EE(Y')$ is isomorphic as $G$-set to~$\E(Y')$.

We shall now see the following.
\begin{equation}\label{eq:sep}
\textit{ $\EE(Y)$   is finitely seperating for $\V(Y)$.}
\end{equation}

\begin{proof}[Proof of~\eqref{eq:sep}]
Consider  $v$, $w \in \V(Y)$.  We shall show
\mbox{$\langle v,_{-}\rangle  \vert_{\EE(Y)} =_{\text a} \langle w,_{-}\rangle   \vert_{\EE(Y)}$.}

For all but finitely many $g$ in a right $H$-transversal in $G$,
\mbox{$\langle v,_{-}\rangle  \vert_{{\mkern-2mu gE_H}} =  \langle w,_{-}\rangle   \vert_{{\mkern-2mu gE_H}}$,} and here
\mbox{$\langle v,g_{-} \rangle \vert_{{E_H}} {=} \langle w,g_{-}\rangle \vert_{{E_H}}$}, hence
\mbox{$g^{-1}\langle v,g_{-} \rangle \vert_{{E_H}} = g^{-1}\langle w,g_{-}\rangle \vert_{{E_H}}$,} hence\newline
 $\langle g^{-1}v,_{-} \rangle\vert_{{E_H}} = \langle g^{-1}w,_{-}\rangle \vert_{{E_H}},$
 hence \mbox{$ \langle g^{-1}v,_{-} \rangle\vert_{{\EE (T_H)}} =  \langle g^{-1}w,_{-}\rangle \vert_{{\EE(T_H)}}$}
by~\eqref{eq:dd},
hence  $g\langle g^{-1}v,_{-} \rangle\vert_{{\EE(T_H)}} = g\langle g^{-1}w,_{-}\rangle \vert_{{\EE (T_H)}},$
hence  \mbox{ $ \langle  v,g_{-} \rangle\vert_{{\EE (T_H)}} =  \langle  w,g_{-}\rangle \vert_{{\EE (T_H)}}$,}
and hence $ \langle  v, _{-} \rangle\vert_{{\mkern-2mu g \EE(T_H)}}
=  \langle  w,_{-}\rangle \vert_{{\mkern-2mu g \EE(T_H)}}$.

For all $g  \in G $,
 $ \langle g^{-1}v, _{-} \rangle \vert_{{\EE (T_H)}}$ and $\langle g^{-1}w, _{-} \rangle \vert_{{\EE(T_H)}}$
  differ\vspace{.5mm}  only on the elements of $\EE(T_{H})$ corresponding
 to the elements of $\E(T_{H})$ crossed by the
reduced $T_{H}$-path    from    $ \pi_{{E_H}}(g^{-1}v)$ to  $ \pi_{{E_H}}(g^{-1}w)$,   hence
 $\langle g^{-1}v,_{-} \rangle\vert_{{\EE (T_H)}} =_{\text a} \langle g^{-1}w,_{-}\rangle \vert_{{\EE (T_H)}},$
hence \mbox{$g\langle g^{-1}v,_{-} \rangle\vert_{{\EE (T_H)}}\hskip-.4pt
=_{\text a}\hskip-.4pt g\langle g^{-1}w,_{-}\rangle \vert_{{\EE (T_H)}}$,}
hence $ \langle  v,g_{-} \rangle\vert_{{\EE (T_H)}} \hskip-.4pt=_{\text a}
 \hskip-.4pt \langle  w,g_{-}\rangle \vert_{{\EE (T_H)}},$
and  hence $ \langle  v, _{-} \rangle \vert_{{\mkern-2mu g \EE(T_H)}}
=_{\text a}  \langle  w,_{-}\rangle  \vert_{{\mkern-2mu g \EE(T_H)}}$.\vspace{.5mm}

This completes the proof of ~\eqref{eq:sep}.
\end{proof}

Consider any $e$, $f \in \E(Y)$ with $e \ne f$.  Since $\delta_{Y}(e^{\ast\ast}) = \{e\}$ and
$\delta_{Y}(f^{\ast\ast}) = \{f\}$, it
 follows that exactly one of the four sets  $$ e^{\ast\ast} \cap f^{\ast\ast},\,\,\,
e^{\ast\ast} \cap f^{\ast\ast\comp},\,\,\,  e^{\ast\ast\comp} \cap f^{\ast\ast},\,\,\,
e^{\ast\ast\comp} \cap f^{\ast\ast\comp}$$
has empty coboundary in $Y$;  we denote that set by $r_{e,f}$.
Thus  $r_{e,f}$ is the vertex set of a union of components of $Y$; also,
$ e^{\ast\ast} $ and $f^{\ast\ast}$ are nested in $\V(Y)$ if and only if $r_{e,f} = \emptyset$.

Set $\mathcal{R} \coloneq \{ r_{e,f} \mid e,f \in \E(Y), e \ne f\} - \{\emptyset\}$.

We shall now prove the following crucial facts.
\begin{align}
&\textit{$\mathcal{R}$ is $G$-quasifree.}\label{eq:R1}
\\&\textit{$\mathcal{R}$ is finitely separating for $\V(Y)$.}\label{eq:R2}
\end{align}

\begin{proof} [Proof of~\eqref{eq:R1} and~\eqref{eq:R2}]
We  form a
$G$-subgraph $X$ of  $\complete(V)$
by adding to $Y$ a $G$-finite $G$-set of edges  that will be specified.  We begin as follows.

Recall that $v_{0} \in \V(E_{H})$, that $S$ is a finite subset of $G$ such that $H \cup S$ generates~$G$,
and that $W$ is a  finite subset of  $V_{{\infty}}$.
We take as our first approximation
$$X \coloneq Y \cup \{(gv_{0},gsv_{0}) \mid g \in G, s \in S - G_{{v_0}}\}
\cup \{ (gv_{0},gw) \mid g \in G,  w \in W\},$$ a $G$-subgraph of $\complete(V)$
obtained by adding to $Y$
 a $G$-finite $G$-set of edges.   In $X/\E(X)$,
each element of $T_{H}$ is identified with $v_{0}$, and the image of $v_{0}$ is stabilized
by $H \cup S$ and, hence, is stabilized by $G$;
also,  each element of $W$ is identified with~$v_{0}$, and each element of $G(W)$
 is then identified with $v_{0}$.  Hence,
 $X/\E(X)$ consists of a single $G$-orbit with a single point, and, therefore, $X$ is connected.

We shall now show that $\EE(Y) \subseteq \BB(X)$.
By \eqref{eq:sep},  $\EE(Y)$ is finitely separating for $\V(Y)\,\, ( = \V(X))$.  Hence,
 for each edge $(v,w)$ in $\E(X)$, there exist only finitely many  $e \in \E(Y)$
such that \mbox{$(v,w) \in \delta_{X}(e^{\ast\ast})$.}
Thus, for each edge $(v,w)$ in \mbox{$\E(X){-}\E(Y)$,} there exist only finitely many  $e \in \E(Y)$
such that \mbox{$(v,w) \in \delta_X(e^{\ast\ast})$.}  Hence, for each edge $(v,w)$ in $\E(X){-}\E(Y)$,
and each   $e \in \E(Y)$, there exist only finitely many $g \in G$ such that
\mbox{$(v,w) \in \delta_X(ge^{\ast\ast})$,} or, equivalently, \mbox{$g^{-1}(v,w) \in \delta_X(e^{\ast\ast})$.}
Since \mbox{$\E(X){-}\E(Y)$} is $G$-finite, and $\delta_Y(e^{\ast\ast}) = \{e\}$, we see that
$  \delta_X(e^{\ast\ast}) $ is finite.
It follows that  $\EE(Y) \subseteq \BB(X)$.

In particular,  for each $e \in \E(Y)$,    $X - \delta_X(e^{\ast\ast})$ has only a finite number of
 components.

Also, $\mathcal{R} \subseteq \langle \EE(Y) \rangle_{{\BB}} \subseteq \BB(X)$.
Since $X$  is connected with $G$-quasifree edge $G$-set,
while $\emptyset \not\in \mathcal{R}$ and $\V(X) \not\in \mathcal{R}$,
we see that each element  of $\mathcal{R}$ has nonempty, finite
coboundary in $X$, and, hence,~\eqref{eq:R1} holds.

We have seen that  each edge in $\E(X) {-} \E(Y)$ lies in
 $\delta_{X}(e^{\ast\ast})$ for only finitely many  $e \in \E(Y)$.
 Since $\E(X) {-} \E(Y)$ is $G$-finite, we see that the $G$-set
$$E' \coloneq \{e \in \E(Y) \mid \delta_{X}(e^{\ast\ast}) \cap (\E(X) {-} \E(Y)) \ne \emptyset \}$$ is $G$-finite.
Notice that $E' = \{e \in \E(Y) \mid \delta_{X}(e^{\ast\ast}) \ne \{e\}  \}$.

In particular,  the  $G$-subset of $\E(Y)$ consisting of
those $e \in \E(Y)$ such that $X - \delta_{X}(e^{\ast\ast})$ has more than
two components is $G$-finite, and, for any such $e$, we may connect every component of
 $X - \delta_{X}(e^{\ast\ast})$ to every other component using a finite set of edges of $\complete(V)$.
   Thus adding to $X$   a suitable $G$-finite $G$-set of edges of $\complete(V)$ ensures
  that,  for each $e \in \E(Y)$,    $X - \delta_{X}(e^{\ast\ast})$ has exactly two components.
This is the final form we want for $X$, and we may assume that this is the $X$ that we had from the start,
and the definition for $E'$ now refers to the new~$X$.

We next prove that $\mathcal{R}$ is $G$-finite, and for this it suffices to prove the
$G$-finiteness  of  the $G$-set consisting
of all the pairs \mbox{$(e,f) \in \E(Y) \times \E(Y)$} such that
$e^{\ast\ast}$ and $f^{\ast\ast}$ are not nested for $\V(X)$.

Consider any such $(e,f)$. In particular, $e \ne f$.

Consider first the case where  $\delta_X(e^{\ast\ast}) = \{e\}$.
 Since  $X - \delta_{X}(f^{\ast\ast})$
has two components, we see that
\mbox{$X - (\delta_{X}(f^{\ast\ast}) \cup \delta_{X}(e^{\ast\ast}))$} has at most three components, and, hence,
$e^{\ast\ast}$ and $f^{\ast\ast}$ are nested for $\V(X)$.

Thus, we may assume that $e$ and $f$  lie in the  $G$-finite $G$-set $E'$, and then it remains to show that
for a given $e$ there are only finitely many possibilities for $f$.
Let $A$ be any finite, connected subgraph of $X$ containing $\delta_{X}(e^{\ast\ast})$.
For any $f \in E'$, for all but finitely many $g \in G$,
$\delta_X(f^{\ast\ast}) \cap gA = \emptyset$, and then
\mbox{$\delta_X((g^{-1}f)^{\ast\ast}) \cap  A = \emptyset$.}
By the $G$-finiteness of $E'$, for all but finitely many $f \in E'$,
 $\delta_{X}(f^{\ast\ast}) \cap A = \emptyset$.
If $\delta_{X}(f^{\ast\ast}) \cap A = \emptyset$, then the connected graph
 $A $
lies entirely in either $f^{\ast\ast}$ or
$f^{\ast\ast\comp}$.  Since $A \supseteq \delta_{X}(e^{\ast\ast})$,
 then  $f^{\ast\ast\comp}\cap \delta_{X}(e^{\ast\ast}) = \emptyset $
or $f^{\ast\ast}\cap \delta_{X}(e^{\ast\ast}) = \emptyset $, respectively.
If $f^{\ast\ast\comp}\cap \delta_{X}(e^{\ast\ast}) = \emptyset $, then, since  $f^{\ast\ast\comp}$ is the
vertex set of a connected subgraph of~$X$, $f^{\ast\ast\comp}$ lies entirely in either
$e^{\ast\ast}$ or
$e^{\ast\ast\comp}$, and, hence, $e^{\ast\ast}$ and $f^{\ast\ast}$ are nested.  A
similar  argument applies if $f^{\ast\ast}\cap \delta_{X}(e^{\ast\ast}) = \emptyset $.

Thus, $\mathcal{R}$ is $G$-finite.  For any edge $(v,w)$ of $X$ and any $r \in \mathcal{R}$,
 there exist only finitely many
$g \in G$ such that $g(v,w) \in   \delta_{X}(r)$, or, equivalently, \mbox{$(v,w) \in \delta_{X}(g^{-1}r)$.}
By the $G$-finiteness of $\mathcal{R}$,
$\langle v,_{-} \rangle \vert_{\mathcal{R}} =_{\text a} \langle w,_{-} \rangle \vert_{\mathcal{R}}$.
Since $X$ is connected, it follows that $\mathcal{R}$ is finitely separating for $\V(X)\,\,(= \V(Y))$.

This completes the proof of~\eqref{eq:R2}.
 \end{proof}

By~\eqref{eq:R1}, $\mathcal{R}$ is $G$-quasifree, and, by~\eqref{eq:R2},
the image of the map $$\V(Y) \to \Map(\mathcal{R},\integers_2), \quad
v \mapsto \langle v,_{-}\rangle\vert_{{\mathcal{R}}},$$
 lies in
 a $G$-stable almost equality class.
Since $G \in \G$, the latter $G$-stable almost equality class is then the vertex $G$-set of some
$G$-tree with $G$-quasifree edge $G$-set, and we let $T_\text{bottom}$ denote such a  $G$-tree.
 Here  we have  a $G$-map
$$\V(Y) \to \V(T_{\text{bottom}}) \subseteq \Map(\mathcal{R},\integers_2),
\quad v \mapsto \langle v,_{-} \rangle \vert_{{\mathcal{R}}}.$$

Consider any $u \in    \V(T_{\text{bottom}})$.
Let $V_{{u}} \subseteq \V(Y)$ denote the fibre over $u$.  Since each element of $\mathcal{R}$ is the vertex set of a union of components of $Y$, we see that
$V_{{u}}$ is the vertex set of some subforest $Y_{{u}}$ of $Y$ that is a union of components of $Y$\!.  In particular,
 each component of $Y$ lies entirely in a  fibre, and we have a fibration of $Y$ into unions of components.

We shall now see the following.
\begin{align}\label{eq:u}
&\textit{There exists some $G_{{u}}$-tree $T_{{u}}$ and some $G_{{u}}$-graph
 map $Y_{{u}} \to T_{{u}}$ which}\\[-.1cm]
&\textit{is bijective on edges.}\nonumber
\end{align}

\begin{proof}[Proof of~\eqref{eq:u}]
By \eqref{eq:sep},  $\EE (Y_{{u}})$ is finitely separating for $\V(Y_{{u}})$.

We claim that  $\EE (Y_{{u}})$ is nested for $\V(Y_{{u}})$.
Consider any $e$, $f \in \E(Y_{{u}})$ with $e \ne f$.  We shall show that $r_{e,f} \cap \V(Y_{{u}}) = \emptyset$.
This is clear if   $r_{e,f} = \emptyset$.  Thus we may suppose that $r_{e,f} \ne \emptyset$,
and, hence $r_{e,f} \in \mathcal{R}$. For each $v \in V_{{u}}$, we see that
 \mbox{$\langle v, _{-} \rangle \vert_{{\mathcal{R}}} = \langle u, _{-} \rangle \vert_{{\mathcal{R}}}  = u$.}
In particular, $\langle v,  r_{e,f} \rangle  = \langle u, r_{e,f} \rangle$.
In particular,  $ \langle \iota_{Y}(e),  r_{e,f} \rangle =\langle u, r_{e,f} \rangle$.
Now recall that $ \langle \iota_{Y}(e),  r_{e,f} \rangle =0$.  Hence
\mbox{$ \langle v,  r_{e,f} \rangle  = \langle u, r_{e,f} \rangle = \langle \iota_Y(e),  r_{e,f} \rangle =0.$}
Thus $r_{e,f} \cap \V(Y_{{u}}) = \emptyset$.  This proves the claim.

Hence   $\EE(Y_{{u}})$ is a nested, finitely separating,
 $G_{{u}}$-subset of $\PP(\V(Y_{{u}}))$.  It follows from Corollary~\ref{cor:Dunwoody} that
$T_{{u}} \coloneq \U(\EE(Y_{{u}}))$ is
a  $G_{{u}}$-tree with $G_{{u}}$-edge set \mbox{$\EE(Y_{{u}}) \simeq \E(Y_{{u}})$} and
there exists a natural
$G_{{u}}$-map $\V(Y_{{u}}) \to \V(\T_{{u}})$.
For each edge $e$ of $Y_{{u}}$, it follows from~\eqref{eq:dd} that
\mbox{$(\langle \iota e,_{-}\rangle\vert_{{\EE(Y)}}) \triangledown
 (\langle \tau e,_{-}\rangle\vert_{{\EE(Y)}}) = \{e^{\ast\ast}\}$.} It can then be seen that
we have a $G_{{u}}$-graph map $Y_{{u}} \to T_{{u}}$ that is bijective on edges;
 it is surjective on vertices if $Y_{{u}}$ is nonempty.  This completes the proof of~\eqref{eq:u}.
\end{proof}

We now create a $G$-tree  denoted $T_\text{middle}$  by $G$-equivariantly  blowing
up each vertex $u$ of $T_{\text{bottom}}$
to the $G_{{u}}$-tree $T_{{u}}$.  Then
$T_\text{middle}$ is a $G$-tree with $G$-quasifree edge $G$-set, and there
is specified  a $G$-graph map $Y \to T_\text{middle}$ which is injective on edges.

Since \mbox{$G \in \G$},
there exists some $G$-tree $T_{G}$ with vertex $G$-set $\V(E_{G})$ and $G$-quasifree edge $G$-set.

We now create a $G$-tree  denoted $T_\text{top}$  by $G$-equivariantly  blowing
 up each vertex $v$ of $T_\text{middle}$ to the $G_{v}$-tree~$T_{G}$  using the
incidence maps for $Y$ to make each element of $\E(Y) \subseteq \E(T_\text{middle})$
incident to appropriate copies  of elements of $\V(Y) \subseteq \V(T_{G})$.
 Then $T_\text{top}$ is a $G$-tree with $G$-quasifree edge $G$-set,
$T_\text{top}$   contains the $G$-forest $Y$ as a $G$-subgraph,
and there is specified a $G$-map $\V(T_\text{top}) \to \V(E_{G})$
 which is the identity map on
$\V(Y)$.

We now make some  adjustments to $T_\text{top}$.

Recall that $Y  =  G(W) \vee\, G(T_{H})$.
We may choose a finite subtree $T_0$  of $T_\text{top}$ which contains the  finite  set
 $\{ v_0\} \cup S v_0  \cup W$, and set $T^+ \coloneq G(T_0) \cup Y$.
Then $T^+$ is a connected $G$-subgraph of $T_\text{top}$.
Now $T^+$ is a  $G$-tree with $G$-quasifree edge $G$-set,
 $T^+$ contains the $G$-forest \mbox{$Y $} as a $G$-subgraph,
 $T^+ - Y$ is $G$-finite,
and there is specified a $G$-map $\V(T^+) \to \V(E_G)$ which is the identity map on
$\V(Y)$.

Then $ T^+\!/\E(Y)$   is a  $G$-finite  $G$-tree with $G$-quasifree edge $G$-set.
While it remains possible, we
successively collapse $G$-orbits of edges of $ T^+$
which become $G$-com\-pressible edges in $T^+\!/\E(Y)$; we thus eventually obtain
 a quotient $G$-tree of $T^+$,  denoted~$T$.
Then $T$ is a  $G$-tree with $G$-quasifree edge $G$-set,
 $T$ contains the $G$-forest~\mbox{$Y$} as a $G$-subgraph,
  $T{-}Y$ is $G$-finite,  $T\!/\E(Y)$ is $G$-incom\-pressible,
and there is specified a $G$-map $\V(T) \to \V(E_{G})$ which is the identity map on
$\V(Y)$.

Recall that $V_{{\infty}} \coloneq \{v \in \V(E_G) - G(\V(E_H)):  G_v \text{ is infinite}\}$
and that  $W $ is  an arbitrary   finite subset of $V_{{\infty}}$.
We shall now prove the following.
\begin{equation}\label{eq:infty}
\textit{$V_{{\infty}}$ is $G$-finite.}
\end{equation}

\begin{proof}[Proof of~\eqref{eq:infty}]  Letting $\overline v_0$ denote
the component of $Y$ containing $v_0$, we may write
 $Y/\E(Y) = G(\overline v_0) \vee G(W)$.

Now
$ T\!/\E(Y)$  is a  $G$-finite, $G$-incompressible $G$-tree
with $G$-quasifree edge $G$-set,
 \mbox{$G(\overline v_0) \vee G(W) \subseteq \V(T\!/\E(Y))$},
and there is specified a $G$-map
$$\V(T\!/\E(Y)) -  G \overline v_0 \to \V(E_{G})$$ which is the identity map on
$G(W)$.

 Let $W'$ be an arbitrary  finite
 subset of $V_\infty$  which contains a $G$-transversal in
 the intersection of $V_\infty$ with the $G$-finite image of
the $G$-map $$\V(T\!/\E(Y)) -  G \overline v_0 \to \V(E_{G}).$$
Then $G(W') \supseteq G(W)$.

The entire foregoing
argument applies with $W'$ in place of $W$, and we get a   $G$-finite,  $G$-incompressible
$G$-tree  $ T'\!/\E(Y)$ with $G$-quasifree edge $G$-set,
 such that \mbox{$G (\overline v_0) \vee G(W') \subseteq \V( T'\!/\E(Y))$}
and there is specified a $G$-map
$$\V(T'\!/\E(Y)) -  G \overline v_0 \to \V(E_{G})$$ which is the identity map on
$G(W')$.

By the choice of $W'$,
each  infinite subgroup of $G$ that stabilizes an element of $\V(T/\E(Y))$
stabilizes  an element of  \mbox{$ G \overline v_0  \vee G(W')  \subseteq \V( T'\!/\E(Y))$.}
Each finite subgroup of $G$  stabilizes  an element of  $\V( T'\!/\E(Y))$.
  Hence $$G\substabs(\V(T\!/\E(Y))) \subseteq G\substabs(\V( T'\!/\E(Y))).$$
By Lemma~\ref{lem:sizes}\eqref{it:1size}, $3\vert G\leftmod  (T\!/\E(Y))\vert
\ge \vert G \leftmod   ( T'\!/\E(Y))  \vert$.
Since $\V(T'\!/\E(Y)) \supseteq G(W')$, we see that
\mbox{$3\vert G\leftmod (T\!/\E(Y)) \vert \ge \vert  G\leftmod G(W')  \vert$.}
Thus, we have  a finite upper bound on the number of $G$-orbits in $V_\infty$.
This completes the proof of~\eqref{eq:infty}.
\end{proof}

By~\eqref{eq:infty},  we may assume that   $W$ is taken to be a $G$-transversal in $V_\infty$ from the start.
Then \mbox{$V_\infty = G(W) \subseteq \V(T)$,} and
 any infinite subgroup of $G$ which stabilizes an element of $\V(E_{G})$ stabilizes an element of $\V(T)$.
Each finite subgroup of $G$  stabilizes  an element of  $\V(T)$.
Thus, $$G\substabs(\V(E_{G})) \subseteq G\substabs(\V(T)),$$ and, hence,
there exists a $G$-map $\phi \colon \V(E_{G}) \to \V(T)$ which is the identity on $ \V(E_{H})$.
We already have a $G$-map   $\psi \colon \V(T) \to \V(E_{G})$ which is the identity  \vspace{-.2mm} on $\V(E_{H})$.
Since $T\!/\E(Y)$ is $G$-incompressible, the composite
\mbox{$\V(T) \overset{\psi}{\to} \V(E_{G}) \overset{\phi}{\to} \V(T)$}
is  bijective, and we may identify $\V(T) $  with a $G$-subset of $\V(E_{G})$
respecting the embeddings of $\V(E_{H})$ in $\V(T) $ and $\V(E_{G})$.  We may then
expand $T$ to a $G$-subtree $T_{G}$ of $\complete(V)$ with
vertex $G$-set $\V(E_{G})$ and edge $G$-set
$$\E(T) \cup \{(v,\phi(v)) \mid v \in \V(E_{G})- \V(T)\}.$$
This completes the proof of Theorem~\ref{thm:first}.
\end{proof}

\section{Countably generable extensions}\label{sec:11}

The following is~\cite[III.8.3]{DD}.

\begin{lemma}\label{lem:3} Let {\normalfont Notation~\ref{not:gen}} hold.
If $\rank (G \text{\normalfont { rel }} H)< \omega_{0}$,  then $E_G - G(E_H)$ is $G$-finite.
\end{lemma}

\begin{proof} Let $S$ be a finite subset of $G$ such that $H \cup S$ generates $G$, and
set \mbox{$F \coloneq  \textstyle \bigcup\limits_{s\in S} ( v_0\triangledown(sv_0))$.}
For each $s$ in the finite set $S$, $sv_0 =_{\text{a}} v_0$, and, hence, $F$~is a finite
subset of $E_G$. Set \mbox{$E' \coloneq G(E_H \cup F)$.}  Then $E'$ is a $G$-subset of $E_G$.
Also, $\langle v_0,_{-}\rangle \vert_{{E-E'}}$ is stabilized by each
$g \in S \cup H$, and, hence, is stabilized by $G$.  Thus $E_G \subseteq E'$,
and then $E_G - G(E_H) \subseteq G(F) $ and  $E_G - G(E_H)$ is $G$-finite, as desired.
\end{proof}

The following is part of the proof of~\cite[III.8.5]{DD}.

\begin{proposition}\label{prop:L} Let {\normalfont Notation~\ref{not:gen}} hold.

Suppose that, for each $g \in G{-}H$, $gE_{H} \cap E_{H} = \emptyset$.

Suppose that   $H \le K \le G$  and
  \mbox{$\rank (K \text{\normalfont { rel }} H) < \omega_{0}$.}

Then there exists some  $L$  such that  \mbox{$H \le K \le L \le G$,}
  \mbox{$\rank (L \text{\normalfont { rel }} H) \le \omega_{0}$,}
and, for each $g \in G{-}L$, $gE_{L} \cap E_{L} = \emptyset$.
\end{proposition}

\begin{proof} We recursively construct an ascending
 sequence $L_{{\bbl1 0{\uparrow}\infty\bbl1 }}$ of subgroups of $G$ such that,
for each $n \in [0{\uparrow}\infty[$\,,
the following hold.
\begin{enumerate}[(1)]
\item $\rank(L_{{n}} \text{ rel } H) < \omega_{0}$.
\item  $\{g \in G \mid g E_{{L_{n}}} \cap E_{{L_{n}}} \ne \emptyset\} \subseteq L_{{n+1}}$.
\end{enumerate}

We set $L_{0} \coloneq K$. Here (1)  holds with $n=0$.

Suppose that we are given some $m \in [0{\uparrow}\infty[$\, and $L_m$
satisfying (1) with $n=m$.
Let $S_{{H}}$ be an $H$-transversal in $ E_{{H}}$.
Let $S_{m}$ be an $L_{m}$-transversal in \mbox{$E_{{L_m}}{-}L_{m}(E_{{H}})$.}
Set \mbox{$F \coloneq \{g \in G \mid g S_{{m}} \cap (S_{{H}} \cup S_{{m}}) \ne \emptyset\}$.}
Set $L_{{m+1}} \coloneq \gen{L_{m} \cup F} \le G$.

Since, for all $g \in G{-}H$, $gE_{{H}} \cap E_{{H}} = \emptyset$,
 we see that $S_{{H}}$ is a   $G$-transversal in $G(E_{{H}})$.
By Lemma~\ref{lem:3}, $S_{m}$ is finite.  Hence \mbox{$G(S_{m}) \cap S_{{H}}$} is finite.
Recall that  $E$ is $G$-quasifree.  Hence
 $F$ is finite.  Thus (1) holds with $n=m+1$.

Consider any $g \in G$ such that $g E_{{L_m}} \cap E_{{L_m}} \ne \emptyset$.
We wish to show that \mbox{$g \in L_{{m+1}}$.}
Notice that $S_{m} \cup S_{{H}}$ is an
$L_{m}$-transversal in $E_{{L_m}}$.  Hence, on replacing $g$ with an element of $L_{m} g L_{m}$,
we may assume that $g(S_m \cup S_{{H}}) \cap (S_m \cup S_{{H}}) \ne \emptyset$.
If \mbox{$g S_{m} \cap (S_{m} \cup S_{{H}}) \ne \emptyset$}, then  \mbox{$g \in F \subseteq L_{m+1}$.}
If  $g S_{{H}}  \cap  S_{m}   \ne \emptyset$, then  $g \in F^{-1} \subseteq L_{m+1}$.
If \mbox{$g S_{{H}} \cap  S_{{H}} \ne \emptyset$,} then
$g E_{{H}} \cap  E_{{H}} \ne \emptyset$ and $g \in H \le L_{{m+1}}$.
Thus (2) holds with $n=m$.

 This completes the recursive construction of~$L_{{\bbl1 0{\uparrow}\infty\bbl1 }}$.

Set $L \coloneq \bigcup\limits_{n \in [0{\uparrow}\infty[} \hskip-6pt L_{n}$.

Then $K =L_{0} \le L$.

By (1),  $\rank(L \text{ rel } H) \le \omega_{0}$.

We have $E_{L} = \bigcup\limits_{n \in [0{\uparrow}\infty[}  E_{{L_{n}}}$.
For any $g \in G$ such that  $g E_{{L}} \cap E_{{L}} \ne \emptyset$, there exist
$m,\, n \in [0{\uparrow}\infty[$\, such that \mbox{$g E_{L_m} \cap E_{L_n} \ne \emptyset$}, and then, by (2),
\mbox{$g \in L_{\max\{m,n\}+1} \le L$.}

 Thus, $L$ has all the desired properties.
\end{proof}

In the remainder of the section we build a corresponding tree $T_L$.

The following is a modification of~\cite[III.8.2]{DD}.

\begin{lemma}\label{lem:K} Let {\normalfont Notation~\ref{not:gen} } hold, and suppose that $H \le K \le G$.

Suppose that \mbox{$\rank (G \text{\normalfont { rel }} H) < \omega_{0}$,}
and that, for each \mbox{$g \in G{-}H$,} $gE_H \cap E_H = \emptyset$.

Suppose that
 \mbox{$\rank (K \text{\normalfont { rel }} H) < \omega_{0}$.}

Suppose that, for each $L$ with
 $K \le L \le G$ and   $\rank (L \text{\normalfont { rel }} K) < \omega_{0}$,
\begin{enumerate}[\normalfont (a)]
\item $L \in \G$, and,
\item  if $LE_K = E_L$, then, for all $g \in L-K$, $gE_K \cap E_K = \emptyset$.
\end{enumerate}

Suppose that $T_H$ is an  $H$-subtree of $\complete(V)$   with vertex $H$-set $\V(E_H)$.

Suppose that $T_K$ is a   $K$-subtree of $\complete(V)$   with vertex $K$-set~$\V(E_K)$ such that \mbox{$T_H \subseteq T_K$.}

Then there exists some $G$-subtree $T_G$ of $\complete(V)$ with vertex $G$-set $\V(E_G)$
such that $T_K \subseteq T_G$.
\end{lemma}

\begin{proof}    We  recursively define a descending sequence
$G_{\bbl1 0{\uparrow} \infty\bbl1 }$ of subgroups of $G$ containing $K$
as follows.  We set   $G_0 \coloneq G$, and, given $n \in [0{\uparrow}\infty[$\, and $G_n$, we define
 $G_{n+1}$ to be the $G_n$-stabilizer of   $\langle v_0, _{-} \rangle\vert_{{E{-}G_n(E_K)}} \in \Map(E{-}G_n(E_K),Z)$.

We set $E_0 = E_G$ and, for each $n \in [0{\uparrow}\infty[$, we set
 $E_{n+1} = G_n(E_K)$.  Then
$E_{\bbl1 0{\uparrow} \infty\bbl1 }$ is a descending sequence of subsets of $E_G$ containing $E_K$.

Consider   $n \in [0{\uparrow}\infty[$. Then $E_{n+1} = G_n(E_K) \subseteq E_{G_n}$.  It may be shown that
$$G_{n} = \{g \in G  : \langle gv_0, _{-} \rangle\vert_{{E-E_n}} = \langle v_0, _{-} \rangle\vert_{{E-E_n}}\},$$
and then that $E_{G_n} \subseteq E_n$.   For each $g \in G{-}G_n$,
$\langle gv_0, _{-} \rangle\vert_{{E-E_n}} \ne \langle v_0, _{-} \rangle\vert_{{E-E_n}}$,
and, hence,    $g\V(E_n) \cap \V(E_n) = \emptyset$, since
\mbox{$V(E_{n}) = \overline \V(E_{n}) \times \{\langle{v_0,_{-}\rangle\vert_{{E-E_{n}}}}\}$.}
In particular, $G_{V(E_n)} = G_n$.

We shall now show the following.
\begin{align}\label{eq:Tn}
&\textit{For each $n \in [0{\uparrow}\infty[$\,, the chain of subsets}
\\[-1.5mm]&\textit{$\V(E_H) \subseteq \V(E_n) \subseteq \V(E_G) \subseteq V$ extends to a chain of subgraphs}  \nonumber
\\[-1.5mm]&\textit{$T_H \subseteq T_{E_n}
\subseteq T^{(n)} \subseteq \complete(V)$ such that  $T^{(n)}$ is  a $G$-tree with}  \nonumber
\\[-.5mm]&\textit{vertex  $G$-set $\V(E_G)$, and
$T_{E_n}$ is a  $G_n$-tree with vertex $G_n$-set $\V(E_n)$.}  \nonumber
\end{align}

\begin{proof}[Proof of~\eqref{eq:Tn}] Notice that $\V(E_0) = \V(E_G)$.
By (a), $G \in \G$.  By Theorem~\ref{thm:first},  there exists some $G$-subtree
 of $\complete(V)$ with vertex $G$-set $\V(E_{G})$ containing~$T_{H}$.
Here we have the desired conditions for $n=0$.

Suppose then that we are given
$n \in [0{\uparrow} \infty[$\,  and $T^{(n)}$ and $T_{E_n}$.  Notice that $G_{{T_{E_n}}} = G_n$.

We have
\begin{align*}
\V(E_n)  &= \overline V(E_n) \times \{  \langle{v_0,_{-}\rangle\vert_{{E-E_{n}}}} \}
=  \overline \V(E_{n+1}) \times \overline \V(E_n{-}E_{n+1})  \times \{  \langle{v_0,_{-}\rangle\vert_{{E-E_{n}}}} \}\\
& =
\textstyle\bigvee\limits_{w \in \overline \V(E_n{-}E_{n+1})\times \{  \langle v_0,_{-}\rangle\vert_{E-E_{n}}\}}
 \hskip-50pt(\overline \V(E_{n+1}) \times \{w\}).
\end{align*}

Now
\mbox{$\overline \V(E_n{-}E_{n{+}1})$} is a $G_n$-stable almost equality class
in $\Map(E_n{-}E_{n{+}1}, Z)$. By (a),
\mbox{$G_n \in  \G$.}
Hence, there exists some $G_n$-tree with vertex $G_n$-set
$\overline \V(E_n{-}E_{n+1})$ and  $G_n$-quasifree edge $G_n$-set.
Equivalently, there exists some $G_n$-tree $\overline T$ with vertex $G_n$-set
$\overline \V(E_n{-}E_{n+1}) \times \{  \langle v_0,_{-}\rangle\vert_{E-E_{n}}\}$
and with $G_n$-quasifree edge $G_n$-set.

Let $w_0\coloneq  \langle{v_0,_{-}\rangle\vert_{{E-E_{n+1}}}} \in \overline \V(E{-}E_{n+1}).$

We now take   $w \in \V(\overline T) = \overline \V(E_n{-}E_{n+1}) \times \{  \langle v_0,_{-}\rangle\vert_{E-E_{n}}\}$
and consider two cases, where in Case 1   \mbox{$w \not\in G_n (w_0)$} and  in Case 2    $w = w_0$,
and here $(G_n)_{w_0} = G_{n+1} \ge K$.

By Lemma~\ref{lem:fg}, in Case 1,   $\rank((G_n)_w) < \omega_{0}$,
while in Case 2, we have $\rank( G_{n+1} \text{ rel } K) < \omega_{0}.$

In Case 1,  by Theorem~\ref{thm:astfg},
 $(G_n)_w \in \G$,
and  there then exists some
$(G_n)_w$-tree
with $(G_n)_w$-quasifree edge $(G_n)_w$-set and
vertex $(G_n)_w$-set $\overline \V(E_{n+1})$.  Equivalently,
there exists some
$(G_n)_w$-tree $T_w$ with $(G_n)_w$-quasifree edge $(G_n)_w$-set and
vertex $(G_n)_w$-set $\overline \V(E_{n+1}) \times \{w\}$.

In Case 2, by (a), $G_{n+1}  \in \G$.
By Theorem~\ref{thm:first},  there exists some $G_{n+1}$-subtree $T_{w_{0}}^-$
 of $\complete(V)$ with vertex $G_{n+1}$-set $ \V(E_{G_{n+1}})$ such that \mbox{$T_H \subseteq T_{w_{0}}^-$.}
Then $T_{w_{0}}^-$ can  be extended to  some $G_{n+1}$-subtree, denoted $T_{w_{0}}$ and $T_{E_{n+1}}$,
 of $\complete(V)$ with vertex $G_{n+1}$-set $ \V(E_{n+1}) = \overline \V(E_{n+1}) \times \{w_0\}$.

We now $G_n$-equivariantly blow up each vertex $w$ of $\overline T$ to $T_w$ and get a  $G_n$-tree~$T$
with $G_n$-vertex set $\V(E_n)$ having a $G_{n+1}$-subtree $T_{E_{n+1}}$ with vertex
$G_{n+1}$-set $\V(E_{n+1})$  such that  $T_H \subseteq T_{E_{n+1}}$.

The $G$-tree $T^{(n)}$ has a  $G_n$-subtree    $T_{E_{n}}$
with vertex $G_n$-set \mbox{$\V(E_n) = \V(T)$.}   We now build $T^{(n+1)}$  from  $T^{(n)}$
by $G$-equivariantly removing the edges in $T_{E_{n}}$ and replacing them
with the edges of the new $G_n$-tree $T$\!, which has the same vertex $G_n$-set as $T_{E_{n}}$.
This completes the proof of~\eqref{eq:Tn}.
\end{proof}

We next show the following.
\begin{equation}\label{eq:stable}
\textit{The descending sequence of subgroups $G_{\bbl1 0{\uparrow} \infty\bbl1 }$ is eventually constant.}
\end{equation}

\begin{proof}[Proof of~\eqref{eq:stable}]
  Here we may assume that all the terms of $G_{\bbl1 0{\uparrow} \infty\bbl1 }$ are  infinite subgroups.

 The $G$-set  $\V(T^{(0)}\!/\E(G(T_H)))$ is   obtained from   $\V(E_G)$ by
identifying all
the elements of $g\V(E_H)$ with each other, for each $g \in G$.
Since $\rank(G \text{ rel } H) < \omega_{0}$, by Lemma~\ref{lem:fg}(iii)
there exists some
$G$-finite $G$-incompressible $G$-tree  $\overline T^{(\omega_{0})}$
such that $\E(\overline T^{(\omega_{0})})$ is $G$-quasifree and
 $$G\substabs(\V(\overline T^{(\omega_{0})}))
 = G\substabs(\V(T^{(0)}\!/\E(G(T_H)))).$$

Consider any $n \in [0{\uparrow}\infty[$.

 The $G$-set  $\V(T^{(n)}\!/\E(G(T_{E_n})))$ is   obtained from   $\V(E_G)$ by
identifying all the elements of $g\V(E_n)$ with each other, for each $g \in G$.
Since $\rank(G \text{ rel } H) < \omega_{0}$, by Lemma~\ref{lem:fg}(iii) there exists some
$G$-finite $G$-incompressible $G$-tree  $\overline T^{(n)}$
such that $\E(\overline T^{(n)})$ is $G$-quasifree and  $$G\substabs(\V(\overline T^{(n)}))
 = G\substabs(\V(T^{(n)}\!/\E(G(T_{E_n})))).$$

Since $\V(E_H) \subseteq   \V(E_n)$, we see that
  there exists a natural $G$-map
 $$\V(T^{(0)}\!/\E(G(T_{H}))) \to \V(T^{(n)}\!/\E(G(T_{E_n}))).$$
Hence, $$G\substabs(\V(T^{(0)}\!/\E(G(T_{H}))))  \subseteq  G\substabs(\V(T^{(n)}/\E(G(T_{E_{n}})))),$$
 and this is equivalent to
$$G\substabs(\V(\overline T^{(\omega_{0})}))  \subseteq   G\substabs(\V(\overline T^{(n)})).$$
By Lemma~\ref{lem:sizes}\eqref{it:1size}, $3\vert G\leftmod   \overline T^{(\omega_{0})} \vert
\ge \vert G \leftmod    \overline  T^{(n)}   \vert$, and we have now shown that,
as $n$ varies over $[0{\uparrow}\infty[$\,,
$\vert G \leftmod     \overline T^{(n)}   \vert$ has a finite bound.
By definition, $\size (\overline T^{(n)})$  is an  element  of
 $\integers[\ttt]$  with non-negative coefficients,
and we now see   that the
sum of the coefficients has a finite
bound.

 Since $V(E_{n+1}) \subseteq \V(E_n)$,  we see that
there exists a natural $G$-map    $$ \V(T^{(n+1)}\!/\E(G(T_{E_{n+1}}))) \to \V(T^{(n)}\!/\E(G(T_{E_{n}}))).$$
Hence,
 $$ G\substabs(\V(T^{(n+1)}/\E(G(T_{E_{n+1}})))) \subseteq G\substabs(\V(T^{(n)}/\E(G(T_{E_{n}})))),$$
  and this is equivalent to
\mbox{$G\substabs(\V(\overline T^{(n+1)}))  \subseteq   G\substabs(\V(\overline T^{(n)}))$.}
Now, by Lemma~\ref{lem:sizes}\eqref{it:3size}, we see that
\mbox{$\size (\overline T^{(n+1)}) \sqsupseteq \size (\overline T^{(n)})$.}
Thus, $\size (\overline T^{(n)})$  increases with $n \in [0{\uparrow}\infty[$\,.
It is not difficult to use the foregoing boundedness restraint to  show that there exists some
$n\in[0{\uparrow}\infty[$\,  such that $\size (\overline T^{(n+1)}) = \size (\overline T^{(n)})$.

By Lemma~\ref{lem:sizes}\eqref{it:2size}, \mbox{$G\substabs(\V(\overline T^{(n+1)}))  = G\substabs(\V(\overline T^{(n)}))$,}
and  this is equivalent to  $$ G\substabs(\V(T^{(n+1)}/\E(G(T_{E_{n+1}}))))= G\substabs(\V(T^{(n)}/\E(G(T_{E_{n}})))).$$
Since $G_{n+1}$ is infinite and the edge $G$-set of $T^{(n+1)}/\E(G(T_{E_{n+1}}))$
is $G$-quasifree, we see that $T^{(n+1)}/\E(G(T_{E_{n+1}}))$ has at most one $G_{n+1}$-stable vertex.
Since  $G_{n+1}$  stabilizes the image of $\V(E_{n+1})$, we see that
the image of $\V(E_{n+1})$ is the  unique $G_{n+1}$-stable vertex of $T^{(n+1)}/\E(G(T_{E_{n+1}}))$.
Since  $T^{(n)}/\E(G(T_{E_{n}}))$ has a $G_{n}$-stable vertex, namely the image of $\V(E_{n})$,
we see that $G_n$ stabilizes some  vertex of $T^{(n+1)}/\E(G(T_{E_{n+1}}))$,
and, as such a vertex is then $G_{n+1}$-stable, it must be the image of $\V(E_{n+1})$.
Thus, $G_{n} \le G_{V(E_{n+1})} = G_{n+1}$.
This completes the proof of~\eqref{eq:stable}.
\end{proof}

By~\eqref{eq:stable}, there exists some  \mbox{$n \in [1{\uparrow}\infty[$}  such that  $G_{n}  = G_{n-1}$.
Hence $$G_nE_K = E_{n+1} \subseteq E_{G_{n}} \subseteq E_n =  G_{n-1}E_K = G_n E_K.$$
Thus \mbox{$G_nE_K = E_{G_n} = E_n$}. By~(b),  for all $g \in G_n{-}K$, $g E_K \cap E_K = \emptyset$.
By (a) and Theorem~\ref{thm:first}, there exists a $G_n$-subtree $T$ of $\complete(V)$
with vertex $G_n$-set $\V(E_n)$ such that~$T_K \subseteq T$.

We now build $T_G$  from  $T^{(n)}$
by $G$-equivariantly removing the edges in $T_{E_{n}}$ and replacing them
with the edges of the new $G_n$-tree $T$, which has the same vertex $G_n$-set as $T_{E_{n}}$.
This completes the proof.
\end{proof}

The following is a modification of~\cite[III.8.4]{DD}.

\begin{theorem}\label{thm:countable} Let {\normalfont Notation~\ref{not:gen}} hold.

Suppose that   $\rank (G \text{\normalfont { rel }} H)\le \omega_{0}$,
and that, for each $g \in G{-}H$, $gE_H \cap E_H = \emptyset$.

Suppose that, for each subgroup $K$ of $G$, if $H \le K$
 and \mbox{$\rank (K \text{\normalfont { rel }} H) < \omega_{0}$,}
then $K \in \G$.

Suppose that  $T_{H}$
is some  $H$-subtree of $\complete(V)$   with vertex $H$-set $\V(E_{H})$\,.

Then there exists some $G$-subtree $T_{G}$ of $\complete(V)$ with vertex $G$-set $\V(E_{G})$
such that $T_{H} \subseteq T_{G}$.
\end{theorem}

\begin{proof}  Let $g_{\bbl1 1{\uparrow}\infty\bbl1 }$ be a sequence in $G$ such that
$H \cup g_{[1{\uparrow}\infty[}$ generates $G$.

We now recursively construct an ascending sequence of subgroups $G_{\bbl1 0{\uparrow}\infty\bbl1 }$ of $G$
  such that $G_0 = H$, and, for each $n \in [1{\uparrow}\infty[$\,, the following hold.
\begin{enumerate}[(1)]
\item $g_n \in G_n$.
\item $\rank(G_n \text{ rel } H) < \omega_{0}$.
\item Whenever $G_n \le K \le G$ and $\rank(K \text{ rel } H) < \omega_{0}$
and $K(E_{G_n}) = E_K$,  then, for each $k \in K{-} G_n$, $kE_{G_n} \cap E_{G_n} = \emptyset$.
\end{enumerate}

Suppose that we are given $n \in [1{\uparrow}\infty[$ and $G_{n-1}$.
Let $\mathbf{K}$ denote the set of those subgroups $K$ of $G$ such that $K$ contains
$G_{n-1} \cup \{g_{n}\}$ and $\rank(K \text{ rel } H) < \omega_{0}$.
By Lemma~\ref{lem:3}, for each $K \in \mathbf{K}$,
$ E_K{ - }K(E_H) $ is $K$-finite. Hence,  \mbox{$K \leftmod(E_K {-} K(E_H))$ } achieves a minimum value
as $K$ ranges over $\mathbf{K}$.  We take $G_{n}$ to be an element of $\mathbf{K}$ where this
minimum is achieved.   Then $G_{n-1} \le G_{n}$, $g_{n} \in G_{n}$, and
$\rank(G_n \text{ rel } H) < \omega_{0}$.  Consider any
subgroup $K$ of $G$ such that $K$ contains $G_{n}$ and $\rank(K \text{ rel } H) < \omega_{0}$
and $K(E_{G_{n}}) = E_K$.  Then $K \in \mathbf{K}$.  By   minimality for $G_n$,
\begin{align*}
\vert G_{n} \leftmod (E_{ G_{n}}  { -}  G_{n}(E_H)) \vert  &\le \vert K \leftmod (E_K {-} K(E_H)) \vert
 \\&= \vert K \leftmod (K (E_{G_{n}})  {-} (K G_n)(E_H)) \vert
\\&\le \vert K \leftmod (K (E_{G_{n}} { - }G_n(E_H))) \vert
 \\&\le \vert G_{n} \leftmod (E_{ G_{n}}  { -}  G_{n}(E_H)) \vert.
\end{align*}
We  have equality throughout, and then $$ K (E_{G_{n}}) { - }(KG_n)(E_H)   = K(E_{G_{n}}  {-} G_n(E_H)),$$ and,
for each $k \in K{-}G_{n}$,
 $k(E_{G_{n}} - G_{n}(E_H)) \,\,\cap\,\, (E_{G_{n}} - G_{n}(E_H)) = \emptyset$, while, by
hypothesis, $kG_{n}(E_H) \cap G_{n}(E_H) = \emptyset$.   Let $S$ be a right $G_n$-transversal in $K$.
Then
\begin{align*}
 K (E_{G_{n}}) &= (K (E_{G_{n}})  {-} (KG_n)(E_H))\,\, \vee\,\, (KG_n)(E_H)
\\&= \textstyle\bigvee\limits_{s \in S} s(E_{G_{n}} - G_{n}(E_H))\,\, \vee\,\, \bigvee\limits_{s \in S} s G_{n}(E_H)
= \textstyle\bigvee\limits_{s \in S} s E_{G_{n}}.
\end{align*}
Hence, for each $k \in K{-}G_{n}$, $kE_{G_n} \cap E_{G_n} = \emptyset$.
This completes the recursive construction of $G_{\bbl1 0{\uparrow}\infty\bbl1 }$.

We next recursively construct an ascending
 sequence $T_{\bbl1 0{\uparrow}\infty\bbl1 }$ of subtrees of $\complete(V)$ containing $T_H$ such
that, for each $n \in [0{\uparrow}\infty[$\,, $T_n$ is a $G_n$-subtree of $\complete(V)$
with vertex $G_n$-set $\V(E_{G_{n}})$.

We take $T_0 \coloneq T_H$. Suppose that we are given $n \in [0{\uparrow}\infty[$\,\, and $T_n$.
By Lemma~\ref{lem:K}, there exists some $G_{n+1}$-subtree of $\complete(V)$ with
vertex $G_{n+1}$-set $\V(E_{G_{n+1}})$ such that $T_n \subseteq T_{n+1}$.
This completes the recursive construction of $T_{\bbl1 0{\uparrow}\infty\bbl1 }$.

We now take $T_G \coloneq \bigcup\limits_{n \in [0{\uparrow}\infty[} T_n$.
\end{proof}

We shall use two different forms of this result.

\begin{corollary}\label{cor:countable}
Let {\normalfont Notation~\ref{not:gen}} hold.

Suppose that   $\vert G \vert \le \omega_{0}$.

Then there exists some $G$-subtree $T_{G}$ of $\complete(V)$ with vertex $G$-set $\V(E_{G})$.
\end{corollary}

\begin{proof}
We take $H=\{1\}$ in Theorem~\ref{thm:countable}.  Here,
 $E_H = \emptyset$ and, for each  subgroup $K$ of $G$, if  $\rank (K) < \omega_{0}$,
then $K \in \G$ by Theorem~\ref{thm:astfg}.
\end{proof}

\begin{corollary}\label{cor2} Let {\normalfont Notation~\ref{not:gen}} hold.

Suppose that   $\rank (G \text{\normalfont { rel }} H)\le \omega_{0}$,
and that, for each $g \in G{-}H$, $gE_H \cap E_H = \emptyset$.

Suppose that every subgroup of $G$ lies in  $\G$.

Suppose that  $T_{H}$
is some  $H$-subtree of $\complete(V)$   with vertex $H$-set $\V(E_{H})$\,.

Then there exists some $G$-subtree $T_{G}$ of $\complete(V)$ with vertex $G$-set $\V(E_{G})$
such that $T_{H} \subseteq T_{G}$.  \hfill\qed
\end{corollary}

\section{The proof}\label{sec:12}

\begin{proof}[Proof of the Almost Stability Theorem~$\ref{thm:ast}$]
We may assume that   Notation~\ref{not:gen}  holds and it suffices to show that
 there exists some $G$-subtree of $\complete(V)$ with vertex $G$-set $\V(E_G)$.

By Corollary~\ref{cor:countable},   we may assume that $\omega_{0} < \vert G \vert$.

By transfinite induction, we may assume that,   for each   subgroup $H$ of $G$, if
$\vert H \vert < \vert G \vert$, then $H \in \G$.

 Set $\gamma \coloneq \vert G \vert$ and choose a bijective map  $\gamma \to G$, $ \beta \mapsto g_{\beta}$.
We shall recursively construct   an ascending chain of subgroups
\mbox{$(G_{{\beta}}  \mid \beta \le \gamma)$} of~$G$ and, at the same time, an
 ascending chain of subtrees $(T_{{\beta}}  \mid \beta \le \gamma)$
of $\complete(V)$.
For each $\beta \le \gamma$,  we shall set $E_{{\beta }} \coloneq E_{{G_{{\beta}}}}$,
$V_{{\beta}} \coloneq \V(E_{{\beta}})$, and the following will hold.
\begin{enumerate}[(1)]
\item $\{g_{\alpha} \mid \alpha < \beta \} \subseteq G_\beta$.
\item $\vert G_{{ \beta}} \vert \le \max\{\omega_{0}, \vert  \beta  \vert \}$.
\item For each $g \in G{ - }G_{{ \beta}}$, $gE_{\beta} \cap E_{\beta} = \emptyset$.
\item  $\V(T_{ \beta}) = V_{{ \beta}}$ and $G_\beta \E(T_\beta) = \E(T_\beta)$.
\end{enumerate}

Suppose  that  we are given some \mbox{$\beta \le \gamma$} and   a chain of subgroups
\mbox{$(G_{{\alpha}}  \mid \alpha < \beta)$} and a chain of subtrees
\mbox{$(T_{{\alpha}}  \mid \alpha < \beta)$}
satisfying  (1)--(4) at each step.

\medskip

\noindent \textbf{Case 1.} $\beta = 0$.

We define $G_{{0}} \coloneq \{1\}$ and $T_{{0}} = \{v_{{0}}\}$.
Here $E_{{0}} = \emptyset$, $V_{{0}} = \{v_{{0}}\}$  and conditions (1)--(4) hold in Case 1.

\medskip

\noindent \textbf{Case 2.} $\beta$ is a successor ordinal.

By Proposition~\ref{prop:L}, there exists some subgroup  $G_{\beta}$ of $G$ with the properties that
\mbox{$G_{{\beta-1}} \cup \{g_{\beta-1}\} \subseteq G_{\beta}$} and
 $\rank(G_{{\beta}} \text{ rel } G_{{\beta-1}}) \le \omega_{0}$ and,
for each $g \in G{-}G_{{\beta}}$,   $g E_{{\beta}} \cap E_{{\beta}} = \emptyset$.
Hence, (3) holds.

Then $\{g_{\alpha} \mid \alpha < \beta \} = \{g_{\alpha} \mid \alpha < \beta {-}1\} \cup \{g_{{\beta -1}}\}
\subseteq G_{{\beta-1}} \cup \{g_{{\beta -1}}\} \subseteq G_{\beta}$,
and (1) holds.

Since  $\rank(G_{{\beta}} \text{ rel } G_{{\beta-1}}) \le \omega_{0}$, we have
$\vert G_{{\beta}} \vert \le \max\{ \omega_{0}, \vert G_{{\beta-1}}  \vert \}$.  Now\linebreak
 $ \vert G_{{\beta}} \vert \le \max\{ \omega_{0}, \vert G_{{\beta-1}}  \vert \}
\le\max\{ \omega_{0},  \vert \beta{-}1  \vert   \} \le  \max\{ \omega_{0},  \vert  \beta  \vert \},$
  and (2) holds.

Since \mbox{$\vert \beta {-} 1 \vert \le \beta {-}1 < \beta \le \gamma$}, we also have
  \mbox{$\vert G_{{\beta}}\vert \le
\max\{ \omega_{0}, \vert  \beta{-}1 \vert   \} < \gamma$,} and then
 every subgroup of $G_{{\beta}}$ lies in $\G$,
 by the transfinite induction hypothesis.
By Corollary~\ref{cor2}, there exists some $G_{\beta}$-subtree $T_{\beta}$ of $\complete(V)$
with vertex $G_{\beta}$-set $V_{{\beta}}$ such that the   $T_{{\beta - 1}} \subseteq T_{\beta}$.
Hence, (4) holds.

Now conditions (1)--(4) hold in Case 2.

\medskip

\noindent \textbf{Case 3.} $\beta$ is a limit ordinal.

Here, we set
 $G_{{\beta}} \coloneq \bigcup\limits_{\alpha < \beta} G_{{\alpha}}$ and
$T_{{\beta}} \coloneq \bigcup\limits_{\alpha < \beta} T_{{\alpha}}$.

Notice that $E_{{\beta}} = \bigcup\limits_{\alpha < \beta} E_{\alpha}$ and
$V_{{\beta}} = \bigcup\limits_{\alpha < \beta} V_{{\alpha}}$.  Hence (4) holds.

For each $\alpha < \beta$, we have $\alpha{+}1 < \beta$ and
$g_{{\alpha}} \in G_{{\alpha+1}} \subseteq G_{{\beta}}$.  Hence (1) holds.

Notice that $\omega_0 \le \vert{\beta}\vert$. Thus
$$ \textstyle \vert G_{{\beta}} \vert  = \vert \bigcup\limits_{\alpha < \beta}   G_{{\alpha}} \vert
\le \sum\limits_{\alpha < \beta}  \vert G_{{\alpha}} \vert
\le  \sum\limits_{\alpha < \beta}    \max\{ \omega_{0},  \vert \alpha \vert\} \le  \sum\limits_{\alpha < \beta}
   \vert \beta \vert  \le    \vert \beta \vert^2
=     \vert \beta \vert;$$
see, for example,~\cite[Theorem~3.5]{Jech} and ~\cite[Lemma~5.2]{Jech}.  Hence (2) holds.

For each $g \in G$, if $g E_{\beta} \cap E_{\beta} \ne \emptyset$, then there exist
$\alpha_{{1}} < \beta $ and $\alpha_{2} < \beta$ such that $g E_{{\alpha_1}}  \cap E_{{\alpha_2}} \ne \emptyset$,
and then $g \in G_{{\max\{\alpha_1 ,\alpha_2\}}}\le G_{\beta}$. Hence (3) holds.

Thus conditions (1)--(4) hold in Case 3.

\medskip

This completes the recursive construction.

By (1), $G_{{\gamma}} = G$.  By (4), $T_{{\gamma}}$ is a $G$-subtree of $\complete(V)$ with
vertex $G$-set~$\V(E_{G})$.  This completes the proof.
\end{proof}

\newpage

\section{Arbitrary extensions}\label{sec:13}

With a similar argument, we get the relative version,~\cite[III.8.5]{DD}.

\begin{theorem} Let {\normalfont Notation~\ref{not:gen}} hold.

Suppose that, for each $g \in G{-}H$, $gE_H \cap E_H = \emptyset$,
and that there exists some $H$-subtree $T_{H}$ of $\complete(V)$   with vertex $H$-set $\V(E_{H})$\,.
Then there exists some $G$-subtree $T_{G}$ of $\complete(V)$ with vertex $G$-set $\V(E_{G})$
such that $T_{H} \subseteq T_{G}$.
\end{theorem}

\begin{proof}   Set $\gamma \coloneq \vert G \vert$, and choose a bijective map  $\gamma \to G$, $ \beta \mapsto g_{\beta}$.
 We shall recursively construct   an ascending chain of subgroups
\mbox{$(G_{{\beta}}  \mid \beta \le \gamma)$} of~$G$ and, at the same time, an
 ascending chain of subtrees $(T_{{\beta}}  \mid \beta \le \gamma)$
of $\complete(V)$.
For each $\beta \le \gamma$,  we shall write $E_{{\beta }} \coloneq E_{{G_{{\beta}}}}$,
$V_{{\beta}} \coloneq \V(E_{{\beta}})$, and the following will hold.
\begin{enumerate}[(1)]
\item $\{g_{\alpha} \mid \alpha < \beta \} \subseteq G_\beta$.
\item For each $g \in G - G_{{\beta}}$, $gE_{\beta} \cap E_{\beta} = \emptyset$.
\item  $\V(T_{ \beta}) = V_{{ \beta}}$ and $G_\beta \E(T_\beta) = \E(T_\beta)$.
\end{enumerate}

Suppose  that  we are given some \mbox{$\beta \le \gamma$} and   a chain of subgroups
\mbox{$(G_{{\alpha}}  \mid \alpha < \beta)$} and a chain of subtrees
\mbox{$(T_{{\alpha}}  \mid \alpha < \beta)$}
satisfying  (1)--(3) at each step.

\medskip

\noindent \textbf{Case 1.} $\beta = 0$.

We define $G_{{0}} \coloneq H$ and $T_{{0}} = T_{{H}}$.
Now  conditions (1)--(3) hold in Case 1.

\medskip

\noindent \textbf{Case 2.} $\beta$ is a successor ordinal.

By Proposition~\ref{prop:L}, there exists some subgroup  $G_{\beta}$ of $G$ with the properties that
\mbox{$G_{{\beta-1}} \cup \{g_{\beta-1}\} \subseteq G_{\beta}$} and
 $\rank(G_{{\beta}} \text{ rel } G_{{\beta-1}}) \le \omega_{0}$ and,
for each $g \in G{-}G_{{\beta}}$,   $g E_{{\beta}} \cap E_{{\beta}} = \emptyset$.
Hence, (2) holds.

Then $\{g_{\alpha} \mid \alpha < \beta \} \subseteq G_{{\beta-1}} \cup \{g_{{\beta -1}}\} \subseteq G_{\beta}$,
and (1) holds.

By Theorem~\ref{thm:ast}, every subgroup of $G_{{\beta}}$ lies in $\G$.
By Corollary~\ref{cor2}, there exists some $G_{\beta}$-subtree $T_{\beta}$ of $\complete(V)$
with vertex $G_{\beta}$-set $V_{{\beta}}$ such that the   $T_{{\beta - 1}} \subseteq T_{\beta}$.
Hence, (3) holds.

Now conditions (1)--(3) hold in Case 2.

\medskip

\noindent \textbf{Case 3.} $\beta$ is a limit ordinal.

Here, we define
 $G_{{\beta}} \coloneq \bigcup\limits_{\alpha < \beta} G_{{\alpha}}$ and
$T_{{\beta}} \coloneq \bigcup\limits_{\alpha < \beta} T_{{\alpha}}$.

Notice that $E_{{\beta}} = \bigcup\limits_{\alpha < \beta} E_{\alpha}$ and
$V_{{\beta}} = \bigcup\limits_{\alpha < \beta} V_{{\alpha}}$.  Hence (3) holds.

For each $\alpha < \beta$, we have $\alpha{+}1 < \beta$ and $g_{{\alpha}} \in G_{{\alpha+1}} \subseteq G_{{\beta}}$.  Hence (1) holds.

For each $g \in G$, if $g E_{\beta} \cap E_{\beta} \ne \emptyset$, then there exist
$\alpha_{{1}} < \beta $ and $\alpha_{2} < \beta$ such that $g E_{{\alpha_1}}  \cap E_{{\alpha_2}} \ne \emptyset$,
and then $g \in G_{{\max\{\alpha_1 ,\alpha_2\}}}\le G_{\beta}$. Hence (2) holds.

Thus conditions (1)--(3) hold in Case 3.

\medskip

This completes the recursive construction.

By (1), $G_{{\gamma}} = G$.  By (3), $T_{{\gamma}}$ is a $G$-subtree of $\complete(V)$ with
vertex $G$-set~$\V(E_{G})$.  Since $T_{0} = T_{H}$, this completes the proof.
\end{proof}

\medskip

\noindent \textsc{Departament de  Matem\`atiques,\newline
Universitat  Aut\`o\-noma de Bar\-ce\-lo\-na,\newline
E-08193 Bellaterra (Barcelona), Spain}

\medskip

\noindent \emph{E-mail address}{:\;\;}\url{dicks@mat.uab.cat}

\noindent \emph{URL}{:\;\;}\url{http://mat.uab.cat/~dicks/}

\end{document}